\documentclass[a4paper,10pt]{amsart}
\usepackage{amsmath,amsthm,amssymb,mathrsfs}

\allowdisplaybreaks

\newtheorem{theorem}{Theorem}[section]
\newtheorem{maintheorem}{Theorem}
\newtheorem{lemma}[theorem]{Lemma}
\newtheorem{proposition}[theorem]{Proposition}
\newtheorem{corollary}[theorem]{Corollary}
\theoremstyle{definition}
\newtheorem{definition}[theorem]{Definition}
\newtheorem{remark}[theorem]{Remark}
\newtheorem{example}[theorem]{Example}

\newcommand{\op}{^{\mathrm{op}}}

\newcommand{\dd}{\mathsf{d}}
\newcommand{\BB}{\mathsf{B}}
\newcommand{\FF}{\mathcal{F}}
\newcommand{\TT}{\mathcal{T}}
\newcommand{\Ss}{\mathcal{S}}
\newcommand{\Hh}{\mathcal{H}}

\newcommand{\HH}{\mathrm{HH}}
\newcommand{\kk}{\Bbbk}
\newcommand{\ot}{\otimes}
\newcommand{\si}{\sigma}
\newcommand{\vphi}{\varphi}
\newcommand{\De}{\Delta}
\newcommand{\pl}{\partial}
\newcommand{\al}{\alpha}
\newcommand{\be}{\beta}
\newcommand{\Hom}{\operatorname{Hom}}
\newcommand{\Ext}{\operatorname{Ext}}
\newcommand{\Tor}{\operatorname{Tor}}
\newcommand{\Ker}{\operatorname{Ker}}
\newcommand{\Ima}{\operatorname{Im}}
\newcommand{\bA}{\bar{A}}
\newcommand{\varphantom}[1]{\mathrel{\phantom{#1}}}
\newcommand{\wt}[1]{\widetilde{#1}}
\newcommand{\sbullet}{\scriptscriptstyle{\bullet}}
\newcommand{\Phom}[1]{\langle\!\langle #1\rangle\!\rangle}
\newcommand{\Hhom}[1]{[\![#1]\!]}
\newcommand{\nan}{\mathbb{N}}
\newcommand{\inn}{\mathbb{Z}}

\newcommand{\SSS}{\mathbb{S}}
\newcommand{\WL}{\mathbb{WL}}
\newcommand{\fc}{\mathsf{fc}}

\numberwithin{equation}{section}
\usepackage{tikz-cd}
\tikzcdset{arrow style=math font}

\begin{document}
	
\title[BV structures on Hochschild cohomology of GWAs]{Batalin--Vilkovisky algebra structures on the Hochschild cohomology of generalized Weyl algebras}

\author[Liyu Liu]{Liyu Liu}
\address{School of Mathematical Sciences, Yangzhou University, No.\ 180 Siwangting Road, 225002 Yangzhou, Jiangsu, China}
\email{lyliu@yzu.edu.cn}

\author[Wen Ma]{Wen Ma}
\address{School of Mathematical Sciences, Yangzhou University, No.\ 180 Siwangting Road, 225002 Yangzhou, Jiangsu, China}
\email{2922117517@qq.com}

\begin{abstract}
	This paper is devoted to the calculation of Batalin--Vilkovisky algebra structures on the Hochschild cohomology of skew Calabi--Yau generalized Weyl algebras. We firstly establish a Van den Bergh duality at the level of complex. Then based on the results of Solotar et al., we apply Kowalzig and Kr\"ahmer's method to the Hochschild homology of generalized Weyl algebras, and translate the homological information into cohomological one by virtue of the Van den Bergh duality, obtaining the desired Batalin--Vilkovisky algebra structures. Finally, we apply our results to quantum weighted projective lines and Podle\'s quantum spheres, and the Batalin--Vilkovisky algebra structures for them are described completely.
\end{abstract}

\maketitle

\section{Introduction}
 
Hochschild cohomology theory dates from the forties of the last century. It is becoming indispensable in many branches of algebra, such as homological algebra, representation theory, deformation theory, operad theory, and so on. Furthermore, Hochschild cohomology characterizes some geometrical information. For example, the famous Hochschild--Kostant--Rosenberg theorem interprets the K\"ahler differential forms, multi-derivations of smooth commutative algebras in terms of Hochschild (co)homology.

During the development of Hochschild cohomology theory, Gerstenhaber made remarkable contributions. He discovered a new structure on the Hochschild cohomology $\HH^*(A)$ for any algebra $A$ in 1960's \cite{Gerstenhaber:cohomology}, and the structure is nowadays called the Gerstenhaber algebra structure. Roughly speaking,  a Gerstenhaber algebra is a graded vector space equipped with a cup product and a Lie bracket, which is somewhat analogous to a graded Poisson algebra. Gerstenhaber's another contribution, is to establish deeply relations between cohomology theory and algebraic deformation theory, together with Schack \cite{Gerstenhaber-Schack:deformation}. Deformation theory, as a part of noncommutative geomery, appears in the fields of algebra, algebraic geometry, differential geometry,  mathematical physics, and so on.

In the past decades, people found a special class of Gerstenhaber algebra in some fields such as theoretic physics,  Poisson geometry, and string topology, which are called Batalin--Vilkovisky algebras. A Gerstenhaber algebra is a Batalin--Vilkovisky algebra if the defining Lie bracket is induced by an operator of order two. What followed is the question: Since the Hochschild cohomology of any algebra is a Gerstenhaber algebra, for which algebras their Hochschild cohomology admits a Batalin--Vilkovisky algebra structure? Many mathematician focus on the question and answer positively in some situations. Tradler verified that the Hochschild cohomology for symmetric Frobenius algebras are all Batalin--Vilkovisky algebras \cite{Tradler:bv-inner-product}. Ginzburg proved the Hochschild cohomology for  Calabi--Yau algebras are also Batalin--Vilkovisky algebras \cite{Ginzburg:CY-alg}. They used distinct manners; however, their thoughts are similar. That is, by virtue of a duality between the Hochschild homology and cohomology, the Connes operator induces the desired operator of order two. Inspired by this thought, Kowalzig and Kr\"ahmer showed that if a skew Calabi--Yau algebra $A$ has a semisimple Nakayama automorphism,  then $\HH^*(A)$ is a Batalin--Vilkovisky algebra, which is a generalization of Ginzburg's result \cite{Kowalzig-Krahmer:BV-twisted-CY}. Coincidentally, Lambre, Zhou and Zimmerman proved that for a Frobenius algebra $A$ with semisimple Nakayama automorphism, $\HH^*(A)$ is also a Batalin--Vilkovisky, generalizing Tradler's result \cite{Lambre-Zhou-Zimmermann:bv-frobenius-semisimple}. In particular, when $A$ is Koszul Calabi--Yau, the Koszul dual $A^!$ is a symmetric algebra, so both $\HH^*(A)$ and $\HH^*(A^!)$ are Batalin--Vilkovisky algebras. Chen, Yang and Zhou proved Rouquier's Conjecture: $\HH^*(A)$ is isomorphic to $\HH^*(A^!)$ as a Batalin--Vilkovisky algebra \cite{Chen-Yang-Zhou:bv-koszul}.

It is noteworthy that the Hochschild cochain complexes are so big that one seldom computes cohomology via them in practice. Therefore, it is a tough task to determine the Batalin--Vilkovisky algebra structure on $\HH^*(A)$ for a concrete $A$. In this paper, we investigate the structures on the Hochschild cohomology for a class of generalized Weyl algebras. Such algebras were introduced by Bavula in the mid-1990's \cite{Bavula:GWA-def}. The class of generalized Weyl algebras contains numerous examples arising from quantum groups and differential operator rings (see \cite{Bavula:GWA-tensor-product} for detail). So far, a lot of algebraic properties of generalized Weyl algebras have been revealed, such as irreducible representations, homological dimensions,  isomorphisms and automorphisms (cf.\ \cite{Bavula:gldim-polyn}, \cite{Bavula:GWA-def}, \cite{Bavula-Jordan:GWA-aut}, \cite{L:gwa-def}). In particular, the Hochschild homology and cohomology for generalized Weyl algebras $A$ over a polynomial algebra in one variety have been computed \cite{Farinati:Hochschid-homology-GWA}, \cite{Solotar:Hochschild-homology-GWA-quantum}. A necessary and sufficient condition for $A$ to be skew Calabi--Yau was given by the first author \cite{L:gwa-def}, and the Nakayama automorphism was also obtained. However, when $A$ is skew Calabi--Yau,  the Batalin--Vilkovisky algebra structure on $\HH^*(A)$ is still unknown. The goal of this paper is to determine the structure.
 
From now on, let $A$ be a generalized Weyl algebra of quantum type over a polynomial algebra in one variety, and $A^e$ be the enveloping algebra of $A$. By \cite{L:gwa-def} or \cite{Solotar:Hochschild-homology-GWA-quantum}, $A$ as a left $A^e$-module has a free resolution $\FF$ of the form
\[
\begin {tikzcd}
A^e  & (A^e)^3 \ar[l, "d_1"'] & (A^e)^4 \ar[l, "d_2"'] & (A^e)^4  \ar[l, "d_3"'] & (A^e)^4 \ar[l, "d_4"'] & \cdots. \ar[l, "d_5"']
\end {tikzcd}
\]
In Section \ref{sec:vdb-duality}, we use the resolution $\FF$ to give an explicit Van den Bergh duality in the language of homotopy category. More specifically, if $A$ satisfies the skew Calabi--Yau condition that was given in \cite{L:gwa-def}, then for any $A$-bimodule $M$, denoting $\Ss(M)=\Hom_{A^e}(\FF, M)$ and $\TT(M)=M^\nu\ot_{A^e}\FF$, there are cochain maps $f$ and $g$ as follows (here $\TT(M)$ is regarded as a cochain complex),
\[
\begin {tikzcd}
\Ss(M) \ar[r, "g", shift left=1] & \TT(M)[-2]. \ar[l, "f", shift left=1] 
\end {tikzcd}
\]
We prove $f$ and $g$ are quasi-isomorphisms by constructing homotopy maps $s\colon fg \Rightarrow 1_{\Ss(M)}$ and $t\colon gf \Rightarrow 1_{\TT(M)[-2]}$.
  
\begin{maintheorem}[Theorem \ref{thm:quasi-inverse}]
	$f$ and $g$ are quasi-inverse mutually, so they are both quasi-isomorphisms.
\end{maintheorem}

Accordingly, we succeed in obtaining the Van den Bergh duality $g\colon \Ss(M) \to \TT(M)[-2]$ at the level of complex. After that, we are going to determine the Batalin--Vilkovisky algebra structure. In terms of generators and relations, $A$ is represented as
\[
\kk\langle x, y, z \,|\, xz=qzx, yz=q^{-1}zy, yx=p(z), xy=p(qz)\rangle,
\]
where $q$ is generic, $p(z)$ is a polynomial in $z$ without multiple roots. Solotar, Su{\'a}rez-{\'A}lvarez and Vivas have calculated the cohomology $H^*(\Ss)$ by the spectral sequence argument \cite{Solotar:Hochschild-homology-GWA-quantum}, where $\Ss=\Ss(A)$. Based on their results, we find cocycle representatives of the bases for the relevant cohomological groups in Section \ref{sec:bv-structure}, and obtain the following theorem in which the notations $\nsim$, $\wt{p}$, $n$, $\ell$ are introduced in the beginning of Section \ref{sec:bv-structure}.

\begin{maintheorem}[Theorem \ref{thm:Per-cohomo}]\label{mainthm:2}
	When $p\nsim z$, the bases for $H^0(\Ss)$, $H^1(\Ss)$, $H^2(\Ss)$ are represented by the following sets of cocycles respectively,
	\begin{description}
		\item[$H^0(\Ss)$] $\{1\}$,
		\item[$H^1(\Ss)$] $\{(x,-y,0)^T\}$,
		\item[$H^2(\Ss)$] $\{(0, \wt{zp'}, -xz, zy)^T, (z^i, \wt{z^i},0,0)^T\,|\, 0\leq i<n, i\neq n-\ell\}$.
	\end{description}
	When $p\sim z$, the bases for $H^0(\Ss)$, $H^1(\Ss)$, $H^2(\Ss)$ are represented by the following sets of cocycles respectively,
	\begin{description}
		\item[$H^0(\Ss)$]  $\{1\}$,
		\item[$H^1(\Ss)$]  $\{(x, -y, 0)^T, (x, y, 2z)^T\}$,
		\item[$H^2(\Ss)$]  $\{(0, \wt{p}, -xz, zy)^T, (1, 1, 0, 0)^T\}$.
	\end{description}
	Besides, the groups $H^j(\Ss)$ for all $j\geq 3$ are trivial.
\end{maintheorem}

We compute the Batalin--Vilkovisky algebra structure on $\HH^*(A)$ in Section \ref{sec:computation-bv}. The section is divided into two subsections, dealing with the cases $p\nsim z$ and $p\sim z$ respectively. In both cases, we apply the same method. That is, we translate the cocycles in Theorem \ref{mainthm:2} into cycles of $\TT(A)$ by the Van den Bergh duality, and the latter are changed into Hochschild cycles by an comparison constructed in Section \ref{sec:computation-bv}. So we employ the manner introduced by Kowalzig and Kr\"ahmer \cite{Kowalzig-Krahmer:BV-twisted-CY}, and hence obtain the required structure.

\begin{maintheorem}[Theorem \ref{thm:main-result-1}]
	If $p\nsim z$, then $\HH^*(A)$ as a Batalin--Vilkovisky algebra has $\{1, \mathfrak{s}, \mathfrak{v}, \mathfrak{u}^i \,|\, 0\leq i< n, i\neq n-\ell \}$ as a basis, where $|1|=0$, $|\mathfrak{s}|=1$, $|\mathfrak{v}|=|\mathfrak{u}^i|=2$, and in addition,
	\begin{enumerate}
		\item $1$ is the identity of $\HH^*(A)$ with respect to the cup product, and the cup products of other pairs of basis elements are trivial,
		\item $\De(\mathfrak{v})=\mathfrak{s}$, and $\De$ acts on other basis elements trivially.
	\end{enumerate}
\end{maintheorem}

\begin{maintheorem}[Theorem \ref{thm:main-result-2}]
	If $p\sim z$, then $\HH^*(A)$ as a Batalin--Vilkovisky algebra has $\{1, \mathfrak{s}, \mathfrak{t}, \mathfrak{u}, \mathfrak{v} \}$ as a basis, where $|1|=0$, $|\mathfrak{s}|=|\mathfrak{t}|=1$, $|\mathfrak{u}|=|\mathfrak{v}|=2$, and in addition,
	\begin{enumerate}
		\item $1$ is the identity of $\HH^*(A)$ with respect to the cup product, and the cup product is zero except $\mathfrak{s}\smallsmile\mathfrak{t}=-\mathfrak{t}\smallsmile\mathfrak{s}=-2\mathfrak{v}$,
		\item $\De(\mathfrak{v})=\mathfrak{s}$, $\De(\mathfrak{t})=-2$, and $\De$ acts on other basis elements trivially.
	\end{enumerate}
\end{maintheorem}

In Section \ref{sec:applications}, we apply the above results to two classes of concrete generalized Weyl algebras. They are quantum weighted projective lines and Podle\'s quantum spheres. Both of them are closely related to the quantum group $SU_q(2)$, and can be realized as subalgebras of $SU_q(2)$. We completely determine the Batalin--Vilkovisky algebra structures on the Hochschild cohomology of them. When we deal with the Podle\'s quantum spheres, there is an interesting phenomenon---the basis for the second Hochschild cohomological group depends on the polynomial $p$ but the dimension is independent. Specifically, for the equatorial Podle\'s quantum sphere, the basis of the second cohomological group is $\{\mathfrak{v}, \mathfrak{u}^1\}$, while the basis is $\{\mathfrak{v}, \mathfrak{u}^0\}$ for others.

\section{Preliminaries}\label{sec:preliminaries}

Throughout this paper, $\kk$ is a field of characteristic $\neq 2$. All vector spaces and algebras are over $\kk$ unless stated otherwise. Unadorned $\Hom$ and $\ot$ mean $\Hom_{\kk}$ and $\ot_{\kk}$, respectively. 

Let us begin with a brief review of Hochschild (co)homology. Let $A$ be an algebra, $A\op$ the opposite of $A$, and $A^e=A\ot A\op$ the enveloping algebra. By regarding $\kk$ as a subspace of $A$ via the unity map $\kk\to A$, we obtain the quotient space $\bA=A/\kk$. The well-known bar complex $(C^\mathrm{bar}_{\sbullet}(A), b')$ is a free resolution of $A$ as a left $A^e$-module defined as follows:
\[
C^\mathrm{bar}_n(A)=A\ot \bA^{\ot n}\ot A, \quad \forall \, n \in \nan,
\]
in which a typical elements $a_0\ot\overline{a_1}\ot\dots\ot\overline{a_n}\ot a_{n+1}$ is written as $a_0[a_1,\dots, a_n]a_{n+1}$, (when $n=0$, we write $a_0\ot a_1$ as $a_0[\,]a_1$), and the differential $b'$ is given by
\begin{align*}
b'(a_0[a_1,\dots,a_n]a_{n+1})&=a_0a_1[a_2,\dots,a_n]a_{n+1}+\sum_{i=1}^{n-1} (-1)^i a_0[a_1,\dots, a_ia_{i+1},\dots, a_n]a_{n+1} \\
&\varphantom{=}{}+(-1)^n a_0[a_1,\dots,a_{n-1}]a_na_{n+1}.
\end{align*}
Such a resolution is called the normalized bar resolution of $A$ in the literature.

Let $M$ be an $A$-bimodule. If we regard it as a right or left $A^e$-module, then we have two complexes $M\ot_{A^e} C^\mathrm{bar}_{\sbullet}(A)$ and $\Hom_{A^e}(C^\mathrm{bar}_{\sbullet}(A), M)$. Both complexes are isomorphic to
\[
\begin {tikzcd}[row sep=3mm, column sep=6mm]
C_{\sbullet}(A, M):  & \!\!\!\!\!\! M  & M\ot \bA \ar[l, "\dd_1"']  & M\ot \bA^{\ot 2} \ar[l, "\dd_2"']  & M\ot\bA^{\ot 3}  \ar[l, "\dd_3"']  & \cdots \ar[l, "\dd_4"'] \\
C^{\sbullet}(A, M): & \!\!\!\!\!\! M \ar[r, "\dd^0"] & \Hom(\bA, M) \ar[r, "\dd^1"]  & \Hom(\bA^{\ot 2}, M)  \ar[r, "\dd^2"]  & \Hom(\bA^{\ot 3}, M)  \ar[r, "\dd^3"]  & \cdots
\end {tikzcd}
\]
respectively, where
\begin{align*}
\dd_n(m[a_1,\dots, a_{n}])&=ma_1[a_2,\dots,a_{n}]+\sum_{i=1}^{n-1} (-1)^i m[a_1,\dots, a_ia_{i+1},\dots, a_{n}] \\
&\varphantom{=}{}+(-1)^{n} a_{n}m[a_1,\dots,a_{n-1}], \\
\dd^nf(a_1,\dots, a_{n+1})&=a_1f(a_2,\dots, a_{n+1})+\sum_{i=1}^{n}(-1)^i f(a_1,\dots, a_ia_{i+1}, \dots, a_{n+1}) \\
&\varphantom{=}{}+(-1)^{n+1}f(a_1,\dots, a_n)a_{n+1}, \quad \forall\, f\in \Hom(\bA^{\ot n}, M).
\end{align*}

\begin{definition} 
	$H_n(A,M)=H_n(C_{\sbullet}(A, M))$ is called the $n$th Hochschild homological group of $A$ with coefficients in $M$, and $H^n(A,M)=H^n(C^{\sbullet}(A, M))$ is called the $n$th Hochschild cohomological group of $A$ with coefficients in $M$. 
\end{definition}

It is direct to see that $H_n(A, M)=\Tor^{A^e}_n(M, A)$ and $H^n(A, M)=\Ext_{A^e}^n(A, M)$. So in practice, people usually compute Hochschild homology or cohomology by using a proper resolution of $A$, which is smaller than the bar resolution. 

We write $\HH^n(A)=H^n(A, A)$ and $\HH^*(A)=\bigoplus_{n\in\nan}\HH^n(A)$. For any automorphism $\nu$ of $A$, the notation $M^\nu$ means the $A$-bimodule whose base space is the same as $M$, and the right module action is twisted by $\nu$, namely,
\[
a\vartriangleright m\vartriangleleft a'=am\nu(a').
\]

\begin{definition}\label{def:skew-CY}
	An algebra $A$ is called $\nu$-skew Calabi--Yau of dimension $d$ for some $d \in\nan$ if
	\begin{enumerate}
		\item $A$ is homologically smooth, i.e., $A$ as a left (or equivalently, right) $A^e$-module, admits a finitely generated projective resolution of finite length,
		\item there are isomorphisms of $A$-bimodules
		\[
		\Ext^i_{A^e}(A, A^e)\cong
		\begin{cases}
		0,  & i\neq d,\\
		A^\nu, & i=d
		\end{cases}
		\]
		in which the regular left module structure on $A^e$ is used for computing the Ext-group, and the right one induces the $A$-bimodule structure on the Ext-group.
	\end{enumerate}
	In this case, $\nu$ is called the Nakayama automorphism of $A$.
\end{definition}

\begin{remark}\label{rmk:hochschild-dimension}
	Nakayama automorphism is unique up to inner, that is, if $\nu_1$ and $\nu_2$ are both Nakayama automorphisms of $A$, then there exists an invertible $u$ in $A$ such that $\nu_1(a)=u\nu_2(a)u^{-1}$ for all $a\in A$. In particular, if the invertible elements of $A$ are exactly the nonzero elements of $\kk$, the Nakayama automorphism is unique.
\end{remark}

The following theorem indicates a duality between Hochschild homology and cohomology for any skew Calabi--Yau algebra.

\begin{theorem}[Van den Bergh duality]\label{thm:VdB-duality}
	Let $A$ be a $\nu$-skew Calabi--Yau of dimension $d$. For any $A$-bimodule $M$ and  integer $i$, there is a  natural isomorphism $H^i(A, M)\cong H_{d-i}(A, M^\nu)$.
\end{theorem}

The above theorem is in fact a special situation of the main theorem proved by Van den Bargh \cite{Van-den-Bergh:VdB-duality}.  Van den Bergh's original proof only requires that $\Ext^d_{A^e}(A, A^e)$ is an invertible bimodule, not necessarily of the form $A^\nu$.  The duality is established in the language of derived category (ibid); however, we use the language of homotopy category instead, for the algebras that we study. 

Next we introduce Gerstenhaber algebras, which originated Gerstenhaber's famous contribution to  Hochschild cohomology theory (see \cite{Gerstenhaber:cohomology} or \cite{Gerstenhaber-Schack:deformation}).

\begin{definition}
	Let $\Hh=\Hh^{\sbullet}$ be a graded vector space, $\smallsmile$ and $[\cdot,\cdot]$ be binary operations on $\Hh$ whose degrees are $0$ and $-1$ respectively. $(\Hh, \smallsmile, [\cdot, \cdot])$ is called a Gerstenhaber algebra if 
	\begin{enumerate}
		\item $(\Hh,\smallsmile)$ is an graded commutative, associative algebra, namely, $a\smallsmile b=(-1)^{|a||b|}b\smallsmile a$ for all homogeneous $a$ and $b$,
		\item $(\Hh[1], [\cdot ,\cdot ])$ is a graded Lie algebra, where $[1]$ means degree shift by $1$.
		\item both operations satisfy the graded Leibniz rule
		\[
		[a, b\smallsmile c]=[a, b]\smallsmile c +(-1)^{(|a|-1)|b|}b\smallsmile [a, c].
		\]
	\end{enumerate}
	The operations $\smallsmile$, $[\cdot, \cdot]$ are called the cup product and the Gerstenhaber bracket respectively.
\end{definition}

The following example is given by Gerstenhaber \cite{Gerstenhaber:cohomology}.

\begin{example}\label{ex:ger4}
	Let $A$ be an algebra. For any $f\in C^m(A,A)$, $g\in C^n(A, A)$, define $f\smallsmile g\in C^{m+n}(A,A)$ and $f\mathbin{\bar{\circ}} g \in C^{m+n-1}(A,A)$ by
	\begin{align*}
	(f\smallsmile g)(a_1, a_2, \dots, a_{m+n})&=(-1)^{mn} f(a_1, a_2, \dots, a_m)g(a_{m+1}, a_{m+2}, \dots, a_{m+n}), \\
	(f\mathbin{\bar{\circ}} g) (a_1, a_2, \dots, a_{m+n-1})&=\sum_{i=1}^m (-1)^{(i-1)(n-1)} f(a_1, \dots, a_{i-1}, g_i, a_{i+n}, \dots, a_{m+n}),
	\end{align*}
	where $g_i=g(a_i, \dots, a_{i+n-1})$. Let $[f,g]=f\mathbin{\bar{\circ}} g-(-1)^{(m-1)(n-1)} g\mathbin{\bar{\circ}} f$.  Then $\smallsmile$ and $[\cdot,\cdot]$ descend to Hochschild cohomology, making $(\HH^*(A), \smallsmile, [\cdot,\cdot])$ into a Gerstenhaber algebra.
\end{example}

Batalin--Vilkovisky algebras are a class of Gerstenhaber algebras, which arise from the BRST theory of topological field theory \cite{Becchi-Rouet-Stora:renormalization}. There are a great deal of interests in these algebras in connection with subjects such as string theory and Poisson geometry (cf.\ \cite{Getzler:bv-topological-field}, \cite{Huebschmann:lie-rinehart}, \cite{Kimura-Voronov-Stasheff:operad-moduli-string}, \cite{Lian-Zukerman:brst-string}, \cite{Xu:bv-poisson-geometry}). 

\begin{definition}
	Let $\Hh=\Hh^{\sbullet}$ be a graded vector space, $\smallsmile$ be a binary operation on $\Hh$ of degree $0$, and $\De$ be an operator on $\Hh$ of order two whose degree is $-1$. Then $(\Hh, \smallsmile, \De)$ is called a Batalin--Vilkovisky algebra, or a BV algebra for short, if
	\begin{enumerate}
		\item $(\Hh,\smallsmile)$ is an graded commutative, associative algebra,
		\item $\De^2=0$.
	\end{enumerate}
\end{definition}

\begin{remark}
	$\De$ is called a BV operator. When $\De(1)=0$, $\De$ is an operator of order two if and only if 
	\begin{align*}
	\De (a\smallsmile b\smallsmile c)&=\De(a\smallsmile b)\smallsmile c+(-1)^{|a|}a\smallsmile\De (b\smallsmile c)+(-1)^{(|a|-1)|b|}b\smallsmile\De (a\smallsmile c) \\
	&\varphantom{=}{}-\De (a)\smallsmile b\smallsmile c-(-1)^{|a|}a\smallsmile \De(b)\smallsmile c-(-1)^{|a|+|b|}a\smallsmile b\smallsmile \De(c)
	\end{align*}
	for all homogeneous $a$, $b$ and $c$.
\end{remark}

A BV algebra is a Gerstenhaber algebra. In fact, for any homogeneous $a$, $b$, by defining
\begin{equation}\label{eq:gerstenhaber-bracket-generator}
[a, b]=(-1)^{|a|}(\De(a\smallsmile b)-\De(a)\smallsmile b-(-1)^{|a|}a\smallsmile \De(b) ),
\end{equation}
we then have a  Gerstenhaber algebra $(\Hh, \smallsmile, [\cdot,\cdot])$. 

\begin{example}
	Hochschild cohomology $\HH^*(A)$ is a BV algebra in each of the following cases:
	\begin{enumerate}
		\item $A$ is a symmetric Frobenius algebra \cite{Tradler:bv-inner-product},
		\item $A$ is a Frobenius algebra with semisimple Nakayama automorphism \cite{Lambre-Zhou-Zimmermann:bv-frobenius-semisimple},
		\item $A$ is a Calabi--Yau algebra \cite{Ginzburg:CY-alg},
		\item $A$ is a skew Calabi--Yau algebra with semisimple Nakayama automorphism \cite{Kowalzig-Krahmer:BV-twisted-CY}.
	\end{enumerate}
\end{example}

In the above example, (1) is a special case of (2), and (3)  is a special case of (4). All of the cases have the common ground that the BV operator is induced by the Connes operator, by virtue of the duality between Hochschild homology and cohomology. For our purpose of this paper, let us briefly introduce some notations and conclusions related to item (4).

Let $A$ be a skew Calabi--Yau algebra whose Nakayama automorphism $\nu$ is semisimple. We assume that $\nu$ is diagonalizable, extending the ground field $\kk$ if necessary. Let $\Lambda$ be the spectrum space of $\nu$, and thus  $A=\bigoplus_{\lambda\in \Lambda} A_{(\lambda)}$. Let $\bA_{(1)}=A_{(1)}/\kk$ and $\bA_{(\lambda)}=A_{(\lambda)}$ ($1\neq \lambda\in \Lambda$), then $\bA=\bigoplus_{\lambda\in \Lambda} \bA_{(\lambda)}$. Denote by $\widehat{\Lambda}$ the monoid generated by $\Lambda$ in $\kk^\times$. For any $\mu\in \widehat{\Lambda}$, let
\[
C^{(\mu)}_n(A, A^\nu)=\bigoplus_{\lambda_i\in \Lambda, \, \prod \lambda_i=\mu} A_{(\lambda_0)}\ot \bA_{(\lambda_1)}\ot \cdots \ot \bA_{(\lambda_n)}.
\]
It is routine the check that $C_{\sbullet}^{(\mu)}(A, A^\nu)$ is a subcomplex of $C_{\sbullet}(A, A^\nu)$ and
\[
C_{\sbullet}(A, A^\nu)=\bigoplus_{\mu\in\widehat{\Lambda}} C_{\sbullet}^{(\mu)}(A, A^\nu).
\] 
In \cite{Kowalzig-Krahmer:BV-twisted-CY} it is proved that $H_*(C_{\sbullet}^{(\mu)}(A, A^\nu))=0$ if $\mu\neq 1$, and hence $H_*(A, A^\nu)=H_*(C_{\sbullet}^{(1)}(A, A^\nu))$.

Similarly, let
\[
C^n_{(\mu)}(A, A)=\Bigl\{f\in C^n(A, A) \Bigm| f(\bA_{(\lambda_1)}\ot \cdots\ot \bA_{(\lambda_n)}) \subset A_{(\mu\lambda_1\cdots\lambda_n)}, \, \forall \, \lambda_i \in \Lambda \Bigr\},
\]
and by convention $A_{(\lambda)}=0$ if $\lambda \notin \Lambda$. We then have the subcomplex $C^{\sbullet}_{(\mu)}(A, A)$ of $C^{\sbullet}(A, A)$, whose $n$th cohomological group is denoted by $\HH^n_{(\mu)}(A)$. Notice that $C^{\sbullet}(A, A)$ cannot be decomposed as the direct sum of these $C^{\sbullet}_{(\mu)}(A, A)$. However, the inclusion $C^{\sbullet}_{(1)}(A, A)\to C^{\sbullet}(A, A)$ is a quasi-isomorphism. Therefore, the Van den Bergh duality $\HH^i(A)\cong H_{d-i}(A, A^\nu)$ is nothing but $\HH^i_{(1)}(A)\cong H_{d-i}(C_{\sbullet}^{(1)}(A, A^\nu))$.

\begin{theorem}[{\cite[\S7]{Kowalzig-Krahmer:BV-twisted-CY}}]\label{thm:kk-bv}
	Let $A$ be a $\nu$-skew Calabi--Yau algebra with $\nu$ semisimple. The Connes operator 
	\begin{align*}
	\BB\colon C_{n}^{(1)}(A, A^\nu) & \longrightarrow C_{n+1}^{(1)}(A, A^\nu) \\
	a_0[a_1, \ldots, a_n] & \longmapsto \sum_{i=0}^n (-1)^{in} 1[a_i,\dots, a_n, a_0, \nu(a_1), \dots, \nu(a_{i-1})]
	\end{align*}
	induces an operator $\De$ on $\HH^*_{(1)}(A)$ via the Van den Bergh duality. Then $(\HH^*(A), \smallsmile, \De)$ is a BV algebra.
\end{theorem}

\begin{remark}
	In fact, $\BB$ is defined on the whole complex $C_{\sbullet}(A, A^\nu)$. Moreover, when $\nu$ is identity map, namely, $A$ is Calabi--Yau, $\BB$ is  the famous Connes operator in cyclic homology theory.
\end{remark}


The notion of generalized Weyl algebra was introduced by Bavula in the mid-1990's. For the general definition of generalized Weyl algebras, please refer to \cite{Bavula:GWA-def}, \cite{Bavula:GWA-tensor-product}. Many researchers in algebra or mathematical physics are interested in a subclass of generalized Weyl algebras which are realized as extensions over a polynomial algebra in one variable. Let us present the definition of them here. 

\begin{definition}\label{def:gwa}
	Let $\si$ be an automorphism of the polynomial algebra $\kk[z]$, $p=p(z)$ be a polynomial of degree at least $1$. Define 
	\[
	A=\kk\langle x, y, z \,|\, xz=\si(z)x, yz=\si^{-1}(z)y, yx=p, xy=\si(p)\rangle,
	\]
	and call it the generalized Weyl algebra determined by $\si$ and $p$.
\end{definition}

Obviously, the basis for $A$ is
\[
\mathcal{B}=\{y^kz^ix^j \,|\, k,i,j\in \mathbb{N}, \, kj=0 \}. 
\]
By the generators and relations, $A$ is a $\mathbb{Z}$-graded algebra, where
\[
A^j=
\begin{cases}
\bigoplus_i \kk z^ix^j, & j\ge 0,\\
\bigoplus_i \kk y^{-j} z^i, & j<0.
\end{cases}
\]

Notice that $\si$ is completely determined by $\si(z)$, which can be written as $qz+c$ for some $q$, $c\in\kk$ and $q\neq 0$. According to \cite{Solotar:Hochschild-homology-GWA-quantum}, up to isomorphism, a generalized Weyl algebra $A$ belongs to exactly one of the three types:
\begin{description}
	\item[commutative] $q=1$, $c=0$. 
	\item[classical] $q=1$, $c\neq 0$. The classical Weyl algebra $A_1(\kk)$ is in this type.
	\item[quantum] $q\neq 1$, $c=0$. The quantum Weyl algebra $\kk\langle x,y \,|\, xy-qyx=1\rangle$ is in this type.
\end{description}

There is a characterization that a generalized Weyl algebra is skew Calabi--Yau.

\begin{theorem}[{\cite[Thm 4.5]{L:gwa-def}}]
	A generalized Weyl algebra $A$ is skew Calabi--Yau algebra if and only if $p$ has no multiple roots. In this situation, the Nakayama automorphism $\nu$ is given by
	\[
	\nu(x)=qx, \;  \nu(y)=q^{-1}y, \; \nu(z)=z,
	\]
	and $A$ is of dimension $2$.
\end{theorem}

By the previous theorem, $\nu$ is diagonalizable, and hence semisimple. So when $p$ has no multiple roots, $\HH^*(A)$ is a BV algebra. Furthermore, the eigenvalues of $\nu$ are $q^j$ for all $j\in\mathbb{Z}$, and the eigenspace $A_{(q^j)}$ corresponding to $q^j$  is $A^j$.

\section{Van den Bergh duality}\label{sec:vdb-duality}

In this section, $A$ is always a skew Calabi--Yau generalized Weyl algebra of quantum type. Since $p=p(z)$ has no multiple roots, we write $p'$ for the formal derivative of $p$, and fix polynomials $\al=\al(z)$, $\be=\be(z)$ such that $\al p+\be p'=1$. 

We will present a nice free resolution $\FF$ of $A$ as a left $A^e$-module. This resolution was constructed in \cite{Solotar:Hochschild-homology-GWA-quantum}, and independently in \cite{L:gwa-def}. In order to make the differentials of $\FF$ clearly expressed, some notations are needed.

For any $\vphi\in\kk[z]$, we write $\wt{\vphi}=\si(\vphi)$. By Definition \ref{def:gwa}, we have $x\vphi=\wt{\vphi}x$ and $\vphi y=y\wt{\vphi}$. Define the linear map $\mathfrak{d}$ to be
\begin{align}
\mathfrak{d} \colon \kk[z] &\longrightarrow \kk[z]\ot\kk[z] \notag\\
1&\longmapsto 0 \notag\\
z^i & \longmapsto \sum_{j=1}^{i}z^{i-j}\ot z^{j-1} \label{eq:nc-der}
\end{align}
for all $i\geq 1$, and we use Sweedler's notation to write $\mathfrak{d}(\vphi)$ as $\sum\vphi_1\ot \vphi_2$, or more simply $\vphi_1\ot \vphi_2$. Since $\kk[z]\ot\kk[z]$ is a subspace of $A^e$, $\vphi_1\ot \vphi_2$ can be viewed as an element of $A^e$ naturally.

Define the left $A^e$-module complex $\FF=(\FF_{\sbullet}, d)$ as follows:
\[
\begin {tikzcd}
A^e  & (A^e)^3 \ar[l, "d_1"'] & (A^e)^4 \ar[l, "d_2"'] & (A^e)^4  \ar[l, "d_3"'] & (A^e)^4 \ar[l, "d_4"'] & \cdots \ar[l, "d_5"']
\end {tikzcd}
\]
where all elements of the free modules are written as row vectors, and the differentials are expressed as matrices:
\begin{align*}
d_1&=
\begin{bmatrix}
x\ot 1-1\ot x \\
y\ot 1-1\ot y \\
z\ot 1-1\ot z
\end{bmatrix},
\\
d_2&=
\begin{bmatrix}
y\ot 1 & 1\ot x & -p_1\ot p_2 \\
1\ot y & x\ot 1 & -q\wt{p_1}\ot \wt{p_2} \\
\wt{z}\ot 1-1\ot z & 0 & q\ot x-x\ot 1 \\
0 & z\ot 1-1\ot \wt{z} & 1\ot y-y\ot q
\end{bmatrix},
\\
d_3&=	
\begin{bmatrix}
x\otimes 1 & -1\otimes x & -\wt{p_1}\ot p_2 & 0 \\
-1\otimes y & y\otimes 1 & 0 & -p_1\ot\wt{p_2} \\
z\otimes1-1\otimes z & 0 & -y\otimes 1 & -1\otimes x \\
0 & \wt{z}\otimes 1-1\otimes\wt{z} & -1\otimes y & -x\otimes 1
\end{bmatrix},
\\
d_4&=
\begin{bmatrix}
y\ot 1 & 1\ot x & -p_1\ot p_2 & 0 \\
1\ot y & x\ot 1 & 0 & -\wt{p_1}\ot \wt{p_2} \\
\wt{z}\ot 1-1\ot z & 0 & -x\ot 1 & 1\ot x \\
0 & z\ot 1-1\ot \wt{z} & 1\ot y & -y\ot 1
\end{bmatrix}.
\end{align*}
For $i\geq 5$, define $d_{i}=d_{i-2}$.

\begin{proposition}[{\cite[\S3.2]{L:gwa-def}} or {\cite[\S4.3]{Solotar:Hochschild-homology-GWA-quantum}}]
	The complex $\FF$ is a free resolution of $A$ as a left $A^e$-module, via the multiplication $A^e\to A$.
\end{proposition}

Since $\FF$ has periodicity two when $i\geq 3$, we say $\FF$ to be the periodical resolution of $A$. For any $A$-bimodule $M$, $\Hom_{A^e}(\FF, M)$, $M^\nu\ot_{A^e}\FF$ are isomorphic to 
\[
\begin {tikzcd}[row sep=3mm]
\Ss^{\sbullet}(M): & M \ar[r, "d^0"] & M^3 \ar[r, "d^1"]  & M^4 \ar[r, "d^2"]  & M^4 \ar[r, "d^3"]  & \cdots, \\
\TT_{\sbullet}(M):  & M  & M^3 \ar[l, "\pl_1"']  & M^4 \ar[l, "\pl_2"']  & M^4 \ar[l, "\pl_3"']  & \cdots, \ar[l, "\pl_4"']
\end {tikzcd}
\]
respectively. We mention that the differentials $d^i$ of $\Ss^{\sbullet}(M)$ share the same matrices form with $d_{i+1}$, and cochains of $\Ss^{\sbullet}(M)$ should be written as column vectors; the differentials $\pl_i$ of $\TT_{\sbullet}(M)$ are induced by $d_i$, by applying $\nu\ot 1$ to all entries, and chains of $\TT_{\sbullet}(M)$ are expressed as row vectors.

Next we will establish the Van den Bergh duality explicitly. For the formal consistency,  let us adapt the chain complex $\TT_{\sbullet}(M)$ to the cochain frame, namely, $\TT^{-i}(M)=\TT_{i}(M)$ and $\pl^{-i}=\pl_i$. Accordingly, cochains of $\TT^{\sbullet}(M)$ are written as column vectors. The matrices forms of $\pl^{-i}$ are listed as follows:
\begin{align*}
\pl^{-1}&=
\begin{bmatrix} q\ot x-x\ot 1 & q^{-1}\ot y-y\ot 1 & 1\ot z-z\ot 1  \end{bmatrix},  \\
\pl^{-2}&=
\begin{bmatrix}
q^{-1}\ot y & y\ot 1 & 1\ot \wt{z}-z\ot 1 & 0 \\
x\ot 1 & q\ot x & 0 & 1\ot z-\wt{z}\ot 1 \\
-p_1\ot p_2 & -q\wt{p_1}\ot\wt{p_2}& x\ot q-q\ot x & y\ot 1-1\ot y
\end{bmatrix},
\\
\pl^{-3}&=
\begin{bmatrix}
q\ot x & -y\ot 1 & 1\ot z-z\ot 1 & 0 \\
-x\ot 1 & q^{-1}\ot y & 0 & 1\ot \wt{z}-\wt{z}\ot 1 \\
-p_1\ot\wt{p_2} & 0 & -q^{-1}\ot y & -y\ot 1 \\
0 & -\wt{p_1}\ot p_2 & -x\ot 1 & -q\ot x
\end{bmatrix},
\\
\pl^{-4}&=
\begin{bmatrix}
q^{-1}\ot y & y\ot 1 & 1\ot \wt{z}-z\ot 1 & 0 \\
x\ot 1 & q\ot x & 0 & 1\ot z-\wt{z}\ot 1 \\
-p_1\ot p_2 & 0 & -q\ot x & y\ot 1 \\
0 & -\wt{p_1}\ot \wt{p_2} & x\ot 1 & -q^{-1}\ot y
\end{bmatrix}.
\end{align*}

With these preparations, we can construct cochain maps between $\Ss^{\sbullet}(M)$ and $\TT^{\sbullet}(M)[-2]$. To this end, let us first construct those between $\Ss^{\sbullet}(A^e)$ and $\TT^{\sbullet}(A^e)[-2]$. 

\begin{lemma}\label{lem:f-g-commutative}
	In the diagram
	\[
	\begin {tikzcd}
	& 0 \ar[r] & A^e \ar[r, "d^0"]\ar[d, "g^0", shift left=1] & (A^e)^3 \ar[r, "d^1"]\ar[d, "g^1", shift left=1]  & (A^e)^4 \ar[r, "d^2"]\ar[d, "g^2", shift left=1]  & (A^e)^4 \ar[r]  & \cdots \\
	\cdots \ar[r] & (A^e)^4 \ar[r, "\pl_{[-2]}^{-1}"] & (A^e)^4 \ar[r, "\pl_{[-2]}^{0}"]\ar[u, "f^0", shift left=1] & (A^e)^3 \ar[r, "\pl_{[-2]}^{1}"]\ar[u, "f^1", shift left=1]  & A^e \ar[r]\ar[u, "f^2", shift left=1]  & 0  
	\end {tikzcd}
	\]
	the two horizontal sequences denote $\Ss^{\sbullet}(A^e)$ and $\TT^{\sbullet}(A^e)[-2]$ respectively. Define the vertical maps by
	\begin{align*}
	g^0&=\begin{bmatrix}
	1\ot \be\\
	-q^{-1}\ot \wt{\be} \\
	q^{-1}\ot \al y \\
	0
	\end{bmatrix},
	\\
	g^1&=
	\begin{bmatrix}
	0 & -q^{-1}\ot \wt{\be} & -q^{-1}\ot \al y \\
	1\ot \be & 0 & 0 \\
	1\ot \al y & 0 & q\wt{p_1}\ot \wt{\be p_2'}-p_1 \ot \be p_2'
	\end{bmatrix},
	\\
	g^2 & =	\begin{bmatrix}	-1\ot \be & q^{-1}\ot \wt{\be} & -q^{-1}\ot \al y & 0	\end{bmatrix},
	\\
	f^0&=	\begin{bmatrix}	p_1\ot p_2 & 0 & q\ot x & -y\ot 1	\end{bmatrix}, \\
	f^1 &=
	\begin{bmatrix}
	0 & \wt{p_1}\ot p_2 & 1\ot x \\
	-qp_1\ot \wt{p_2} & 0 & -y\ot 1 \\
	-q\ot x & y\ot 1 & 0
	\end{bmatrix},
	\\
	f^2&=
	\begin{bmatrix}
	0 \\
	q\wt{p_1}\ot\wt{p_2} \\
	-q\otimes x \\
	y\otimes q
	\end{bmatrix},
	\end{align*}
	then $g$ and $f$ are both cochain maps of right $A^e$-modules.
\end{lemma}

\begin{proof}
	It is sufficient to check the commutativity of each square. Here we only prove $g^1d^0=\pl_{[-2]}^{0} g^0$ and $f^1\pl_{[-2]}^{0}=d^0 f^0$, since other equations can be proved similarly.
	
	First of all, by \eqref{eq:nc-der}, we have
	\[
	\biggl(\biggl(1\ot \frac{\mathrm{d}}{\mathrm{d} z}\biggr) \mathfrak{d}(z^i) \biggr)(z\ot 1-1\ot z)=\mathfrak{d}(z^i)-1\ot (z^i)',
	\]
	and thus
	\begin{align*}
	(p_1\ot p_2')(z\ot 1-1\ot z)&=p_1\ot p_2-1\ot p', \\
	(q\wt{p_1}\ot \wt{p_2'})(z\ot 1-1\ot z)&=\wt{p_1}\ot \wt{p_2}-1\ot \wt{p'}.
	\end{align*}
	
	Next we prove $g^1d^0=\pl_{[-2]}^{0} g^0$. By the two equations above as well as $\al p+\be p'=1$, we have
	\begin{align*}
	g^1 d^0& =
	\begin{bmatrix}
	0 & -q^{-1}\ot \wt{\be} & -q^{-1}\ot \al y \\
	1\ot \be & 0 & 0 \\
	1\ot \al y & 0 & q\wt{p_1}\ot \wt{\be p_2'}-p_1\ot \be p_2'
	\end{bmatrix}
	\begin{bmatrix}
	x\ot 1-1\ot x \\
	y\ot 1-1\ot y \\
	z\ot 1-1\ot z
	\end{bmatrix}
	\\
	&=
	\begin{bmatrix}
	-q^{-1}y\ot \wt{\be} +q^{-1}\ot y\wt{\be} -q^{-1}z\ot \al y+q^{-1}\ot z\al y \\
	x\ot\be-1\ot x\be \\
	x\ot \al y-1\ot x\al y-(p_1 \ot \be p_2')(z\ot 1-1\ot z)+\!(q\wt{p_1}\ot \wt{\be p_2'})(z\ot 1-1\ot z) 
	\end{bmatrix}
	\\
	&=
	\begin{bmatrix}
	-q^{-1}y\ot \wt{\be} +q^{-1}\ot y\wt{\be} -q^{-1}z\ot \al y+q^{-1}\ot z\al y \\
	x\ot\be-1 \ot x\be \\
	x\ot \al y-1\ot x\al y-(1\ot\be)(p_1 \ot  p_2-1\ot p')+(1\ot\wt{\be})(\wt{p_1}\ot \wt{p_2}-1\ot\wt{p'})
	\end{bmatrix}
	\\
	&=
	\begin{bmatrix}
	-q^{-1}y\ot \wt{\be} +q^{-1}\ot y\wt{\be} -q^{-1}z\ot \al y+q^{-1}\ot z\al y \\
	x\ot\be-1\ot x\be \\
	x\ot \al y-1\ot \wt{\al p}-p_1 \ot  \be p_2+1\ot \be p'+\wt{p_1}\ot \wt{\be p_2}-1 \ot\wt{\be p'}
	\end{bmatrix}
	\\
	&=
	\begin{bmatrix}
	-q^{-1}y\ot \wt{\be} +q^{-1}\ot y\wt{\be} -q^{-1}z\ot \al y+q^{-1}\ot z\al y \\
	x\ot\be-1\ot x\be \\
	x\ot \al y-p_1 \ot  \be p_2-1\ot \al p+\wt{p_1}\ot \wt{\be p_2}
	\end{bmatrix},
	\\
	\pl_{[-2]}^{0} g^0 &=
	\begin{bmatrix}
	q^{-1}\ot  y & y\ot 1 & 1\ot \wt{z}-z\ot 1 & 0 \\
	x\ot 1 & q\ot x & 0 & 1\ot z-\wt{z}\ot 1 \\
	-p_1\ot p_2 & -q\wt{p_1}\ot\wt{p_2} & x\ot q-q\ot x & y\ot 1-1\ot y
	\end{bmatrix}
	\begin{bmatrix}
	1\ot \be\\
	-q^{-1}\ot \wt{\be} \\
	q^{-1}\ot \al y \\
	0
	\end{bmatrix}
	\\
	&=
	\begin{bmatrix}
	q^{-1}\ot \be y-q^{-1}y\ot \wt{\be}-q^{-1}z\ot \al y+q^{-1}\ot z\al y \\
	x\ot\be-1\ot \wt{\be}x \\
	-p_1\ot \be p_2+\wt{p_1}\ot \wt{\be p_2}+x\ot \al y-1\ot \al yx
	\end{bmatrix}
	\\
	&=g^1 d^0.
	\end{align*}
	
	The final step is to show $f^1\pl_{[-2]}^{0}=d^0 f^0$. In fact,
	\begin{align*}
	f^1\pl_{[-2]}^{0}& =\begin{bmatrix}
	0 & \wt{p_1}\ot p_2 & 1\ot x \\
	-qp_1\ot \wt{p_2} & 0 & -y\ot 1 \\
	-q\ot x & y\ot 1 & 0 \\
	\end{bmatrix}
	\\
	&\varphantom{=}{}\cdot\begin{bmatrix}
	q^{-1}\ot y & y\ot 1 & 1\ot \wt{z}-z\ot 1 & 0 \\
	x\ot 1 & q\ot x & 0 & 1\ot z-\wt{z}\ot 1 \\
	-p_1\ot p_2 & -q\wt{p_1}\ot\wt{p_2}& x\ot q-q\ot x & y\ot 1-1\ot y
	\end{bmatrix}
	\\
	&=
	\begin{bmatrix}
	\wt{p_1}x \ot p_2 -p_1\ot p_2x & 0 & -q\ot x^2+qx\ot x & -\wt{p}\ot 1+y\ot x \\
	-p_1\ot y\wt{p_2}+yp_1 \ot p_2 & 0 & -q\ot \wt{p}+qy\ot x & -y^2\ot 1+y\ot y \\
	-1\ot yx+yx\ot 1 & 0 & -q\ot \wt{z}x+qz\ot x & y\ot z-y\wt{z}\ot 1
	\end{bmatrix},
	\\
	d^0 f^0 &=
	\begin{bmatrix}
	x\ot 1-1\ot x \\
	y\ot 1-1\ot y \\
	z\ot 1-1\ot z
	\end{bmatrix}
	\begin{bmatrix}
	p_1\ot p_2 & 0, q\ot x & -y\ot 1
	\end{bmatrix}
	\\
	&=
	\begin{bmatrix}
	xp_1\ot p_2-p_1\ot p_2x & 0 & qx\ot x-q\ot x^2 & -xy\ot 1+y\ot x \\
	yp_1\ot p_2-p_1\ot p_2y & 0 & qy\ot x-q\ot xy & -y^2\ot 1+y\ot y \\
	-1\ot p+p\ot 1 & 0 & qz\ot x-q\ot xz & -zy\ot 1+y\ot z
	\end{bmatrix}
	\\
	&=f^1\pl_{[-2]}^{0}. \qedhere
	\end{align*}
\end{proof}

\begin{proposition}\label{prop:quasi-inverse}
	In Lemma \ref{lem:f-g-commutative}, $f$ is quasi-inverse to $g$. Hence $f$ and $g$ are both invertible as morphisms of the homotopy category $\mathbf{K}(A^e)$.
\end{proposition}

\begin{proof}
	It suffices to construct homotopy maps $s\colon fg \Rightarrow 1_{\Ss^{\sbullet}(A^e)}$ and $t\colon gf \Rightarrow 1_{\TT^{\sbullet}(A^e)[-2]}$.	Define $s$ as follows:
	\[
	\begin {tikzcd}[row sep=10mm, column sep=12mm]
	A^e \ar[r, "d^0"]\ar[d, "f^0g^0"', shift right=1]\ar[d, "1", shift left=1] & (A^e)^3 \ar[r, "d^1"]\ar[d, "f^1g^1"', shift right=1]\ar[d, "1", shift left=1]\ar[ld, "s^1"']  & (A^e)^4 \ar[r, "d^2"]\ar[d, "f^2g^2"', shift right=1]\ar[d, "1", shift left=1]\ar[ld, "s^2"']  & (A^e)^4 \ar[r, "d^3"]\ar[d, "0"', shift right=1]\ar[d, "1", shift left=1]\ar[ld, "s^3"'] & (A^e)^4 \ar[r, "d^4"]\ar[d, "0"', shift right=1]\ar[d, "1", shift left=1]\ar[ld, "s^4"']   & \cdots \\
	A^e \ar[r, "d^0"] & (A^e)^3 \ar[r, "d^1"]  & (A^e)^4 \ar[r, "d^2"]  & (A^e)^4 \ar[r, "d^3"]  & (A^e)^4 \ar[r, "d^4"]  & \cdots
	\end {tikzcd}
	\]
	where
	\begin{align*}
	s^1&=	\begin{bmatrix}	0 & 0 & p_1\ot \be p_2'	\end{bmatrix}, 
	\\
	s^2&=
	\begin{bmatrix}
	0 & 0 & \wt{p_1}\ot \be p_2' & 0 \\
	-1\ot \al y & 0 & 0 & p_1\ot \wt{\be p_2'} \\
	1\ot \be & 0 & 0 & 0
	\end{bmatrix},\\
	s^3&=
	\begin{bmatrix}
	0 & 0 & p_1\ot \be p_2' & 0 \\
	1\ot \al y & 0 & 0 & \wt{p_1}\ot \wt{\be p_2'} \\
	1\ot \be & 0 & 0 & 0 \\
	0 & 1\ot \wt{\be} & 1\ot \al y & 0
	\end{bmatrix},
	\\
	s^4&=
	\begin{bmatrix}
	0 & 0 & \wt{p_1}\ot \be p_2' & 0 \\
	-1\ot \al y & 0 & 0 & p_1\ot \wt{\be p_2'} \\
	1\ot \be & 0 & 0 & 0 \\
	0 & 1\ot \wt{\be} & -1\ot \al y & 0
	\end{bmatrix},
	\end{align*}
	and as before, $s^i=s^{i-2}$ for all $i\geq 5$.
	
	By the operations of matrices, we have
	\begin{align*}
	f^0g^0&=p_1\ot \be p_2+1\ot \al p \\
	&=1_M+s^1d^0, \\
	f^1g^1&=
	\begin{bmatrix}
	\wt{p_1}\ot \be p_2+1\ot \al p & 0 & q\wt{p_1}\ot x\be p_2'-p_1\ot \be p_2'x \\
	-y\ot \al y &  p_1\ot \wt{\be p_2} &  p_1\ot \al p_2y+yp_1\ot \be p_2'-qp_1y\ot \wt{\be p_2'} \\
	y\ot \be & 1\ot x\be & 1\ot \al p
	\end{bmatrix}
	\\
	&=1_{M^3}+s^2d^1+d^0s^1, \\
	f^2g^2&=
	\begin{bmatrix}
	0 & 0 & 0 & 0 \\
	-q\wt{p_1}\ot \be \wt{p_2} & \wt{p_1}\ot \wt{\be p_2} & -\wt{p_1}\ot \al p_2 y & 0 \\
	q\ot \be x & -1\ot x\be & 1\ot \al p & 0 \\
	-y\ot q\be & y\ot \wt{\be} & -y\ot \al y & 0 
	\end{bmatrix}
	\\
	&=1_{M^4}+s^3d^2+d^1s^2.
	\end{align*}
	Furthermore, $1_{M^4}+s^{i+1}d^{i}+d^{i-1}s^{i}=0$ for all $i\geq 3$. We thus conclude that $fg$ is homotopic to $1_{\Ss^{\sbullet}(A^e)}$.
	
	Next, define $t$ as follows:
	\[
	\begin {tikzcd}[row sep=10mm, column sep=12mm]
	\cdots \ar[r] & (A^e)^4 \ar[r, "\pl_{[-2]}^{-1}"] & (A^e)^4 \ar[r, "\pl_{[-2]}^{-1}"] & (A^e)^4 \ar[r, "\pl_{[-2]}^{0}"] & (A^e)^3 \ar[r, "\pl_{[-2]}^{1}"]  & A^e    \\
	\cdots \ar[r] & (A^e)^4 \ar[r, "\pl_{[-2]}^{-1}"']\ar[u, "0", shift left=1]\ar[u, "1"', shift right=1] & (A^e)^4 \ar[r, "\pl_{[-2]}^{-1}"']\ar[u, "0", shift left=1]\ar[u, "1"', shift right=1]\ar[lu, "t^{-1}"] & (A^e)^4 \ar[r, "\pl_{[-2]}^{0}"']\ar[u, "g^0f^0", shift left=1]\ar[u, "1"', shift right=1]\ar[lu, "t^0"] & (A^e)^3 \ar[r, "\pl_{[-2]}^{1}"']\ar[u, "g^1f^1", shift left=1]\ar[u, "1"', shift right=1]\ar[lu, "t^1"]  & (A^e) \ar[u, "g^2f^2", shift left=1]\ar[u, "1"', shift right=1]\ar[lu, "t^2"]
	\end {tikzcd}
	\]
	where
	\begin{align*}
	t^2&=\begin{bmatrix}
	0 \\
	0 \\
	-q\wt{p_1}\ot \wt{\be p_2'}
	\end{bmatrix},
	\\
	t^1&=
	\begin{bmatrix}
	0 & 0 & 0 \\
	0 & -q^{-1}\ot \al y & q^{-1}\ot \wt{\be} \\
	-p_1\ot \wt{\be p_2'} & 0 & 0 \\
	0 & -\wt{p_1}\ot \be p_2' & 0
	\end{bmatrix},
	\\
	t^0&=
	\begin{bmatrix}
	-q^{-1}\ot \al y & 0 & 1\ot \wt{\be} & 0 \\
	0 & 0 & 0 & 1\ot \be \\
	-p_1\ot \be p_2' & 0 & 0 & 0 \\
	0 & -\wt{p_1}\ot \wt{\be p_2'} & 0 & q^{-1}\ot \al y
	\end{bmatrix},
	\\
	t^{-1}&=
	\begin{bmatrix}
	0 & 0 & 1\ot \be & 0 \\
	0 & -q^{-1}\ot \al y & 0 & 1\ot \wt{\be} \\
	-p_1\ot \wt{\be p_2'} & 0 & q^{-1}\ot \al y & 0 \\
	0 & -\wt{p_1}\ot \be p_2' & 0 & 0
	\end{bmatrix},
	\end{align*}
	and $t^i=t^{i+2}$ for all $i\leq -2$.
	
	The following procedure is almost a copy of what we just did for $s$. We can directly verify that
	\begin{align*}
	g^2f^2&=\wt{p_1}\ot \wt{p_2\be}+1\ot \wt{\al p} \\
	&=1_M+\pl_{[-2]}^{1}t^2,
	\\
	g^1f^1&=
	\begin{bmatrix}
	p_1\ot \wt{\be p_2}+1\ot \wt{\al p} & -q^{-1}y\ot \al y & \!\! q^{-1}y\ot \wt{\be} \\
	0 & \wt{p_1}\ot p_2\be  & 1\ot x\be \\
	qp_1\ot x\be p_2'-q^2\wt{p_1}\ot x\wt{\be p_2'} & \wt{p_1}\ot p_2\al y+q\wt{p_1}y\ot \wt{\be p_2'}-p_1y\ot \be p_2' & 1\ot \wt{\al p} 
	\end{bmatrix}
	\\
	&=1_{M^3}+t^2\pl_{[-2]}^{1}+\pl_{[-2]}^{0}t^1,
	\\
	g^0f^0&=
	\begin{bmatrix}
	p_1\ot p_2\be & 0 & q\ot x\be & -y\ot \be \\
	-q^{-1}p_1\ot p_2\wt{\be} & 0 & -1\ot x\wt{\be} & q^{-1}y\ot \wt{\be} \\
	q^{-1}p_1\ot p_2\al y & 0 & 1\ot \wt{\al p} & -q^{-1}y\ot \al y \\
	0 & 0 & 0 & 0
	\end{bmatrix}
	\\
	&=1_{M^4}+t^1\pl_{[-2]}^{0}+\pl_{[-2]}^{-1}t^0,
	\end{align*}
	and $1_{M^4}+t^{i+1}\pl_{[-2]}^{i}+\pl_{[-2]}^{i-1}t^{i}=0$ for all $i\leq -1$. Consequently, $gf$ is homotopic to $1_{\TT^{\sbullet}(A^e)[-2]}$.
\end{proof}

Next let us establish quasi-isomorphisms between $\Ss^{\sbullet}(M)$ and $\TT^{\sbullet}(M)[-2]$, according to what we did for $M=A^e$. Notice that $\Ss^{\sbullet}(M)\cong \Ss^{\sbullet}(A^e)\ot_{A^e}M$ and $\TT^{\sbullet}(M)[-2]\cong \TT^{\sbullet}(A^e)[-2]\ot_{A^e} M$. If we apply the functor $-\ot_{A^e}M$ to the diagram given in Lemma \ref{lem:f-g-commutative}, then we immediately obtain the corresponding cochain maps which are still denoted by $f$ and $g$. Since $-\ot_{A^e}M$ preserves homotopy, both $f$ and $g$ are quasi-isomorphisms by Proposition \ref{prop:quasi-inverse}. Therefore, we have

\begin{theorem}\label{thm:quasi-inverse}
	There are quasi-isomorphisms $f \colon \TT^{\sbullet}(M)[-2] \to \Ss^{\sbullet}(M)$ and $g\colon \Ss^{\sbullet}(M)\to \TT^{\sbullet}(M)[-2]$ that fit the following commutative diagram
	\[
	\begin {tikzcd}
	& 0 \ar[r] & M \ar[r, "d^0"]\ar[d, "g^0", shift left=1] & M^3 \ar[r, "d^1"]\ar[d, "g^1", shift left=1]  & M^4 \ar[r, "d^2"]\ar[d, "g^2", shift left=1]  & M^4 \ar[r]  & \cdots \\
	\cdots \ar[r] & M^4 \ar[r, "\pl_{[-2]}^{-1}"] & M^4 \ar[r, "\pl_{[-2]}^{0}"]\ar[u, "f^0", shift left=1] & M^3 \ar[r, "\pl_{[-2]}^{1}"]\ar[u, "f^1", shift left=1]  & M \ar[r]\ar[u, "f^2", shift left=1]  & 0  
	\end {tikzcd}
	\]
	for all $A$-bimodules $M$. Furthermore, $f$ and $g$ are mutually quasi-inverse.
\end{theorem}

Theorem \ref{thm:quasi-inverse} in fact enables us to establish the Van den Bergh duality at the level of  complex.

\begin{remark}\label{rmk:g-independence}
	In Theorem \ref{thm:quasi-inverse}, $g$ depends upon the choice of $\al$ and $\be$. However, the isomorphisms of cohomological groups induced by $g$ are independent of $\al$ or $\be$, since $H^*(g)$ is the inverse of $H^*(f)$, and the latter is independent of $\al$ or $\be$.
\end{remark}

\section{Cohomology of periodical complexes}\label{sec:bv-structure}

In this section, $A$ still denotes a skew Calabi--Yau generalized Weyl algebra of quantum type. In addition,  assume $q$ is generic, i.e., $q$ is not a root of unity. We write
\[
p=\sum_{i=0}^{n}a_iz^{n-i}, \qquad a_0\neq 0,
\]
and $\ell=\min\{ j \,|\, ja_j\neq 0 \}$, $\Ss=(\Ss^{\sbullet}(A), d)$, $\TT=(\TT_{\sbullet}(A), \pl)$. For two polynomials $\vphi$ and $\psi\in\kk[z]$, we use $\vphi\sim\psi$ to indicate that there is a nonzero $c\in\kk$ such that $\vphi=c\psi$.

In \cite{Solotar:Hochschild-homology-GWA-quantum}, by applying the spectral sequence argument, the authors succeeded in computing the dimensions of $H^*(\Ss)$. 

\begin{theorem}[{\cite[Thm 1.1]{Solotar:Hochschild-homology-GWA-quantum}}]\label{thm:Per-cohomo-1}
	The nontrivial cohomological groups are $H^0(\Ss)$, $H^1(\Ss)$ and $H^2(\Ss)$. More precisely,
	\begin{enumerate}
		\item if $p\nsim z$, then $\dim H^0(\Ss)=1$, $\dim H^1(\Ss)=1$,  $\dim H^2(\Ss)=n$;
		\item if $p\sim z$, then $\dim H^0(\Ss)=1$,  $\dim H^1(\Ss)=2$, $\dim H^2(\Ss)=2$.
	\end{enumerate}
\end{theorem}

In order to find the BV algebra structure of $\HH^*(A)$, we need to know bases for these groups. By the computation in \cite{Solotar:Hochschild-homology-GWA-quantum}, we can lift the bases for $H^0(\Ss)$ and $H^1(\Ss)$ easily, but for $H^2(\Ss)$, the lifting is boring and messy.  So we will give an enhanced version of Theorem \ref{thm:Per-cohomo-1}. 

\begin{theorem}\label{thm:Per-cohomo}
	When $p\nsim z$, the bases for $H^0(\Ss)$, $H^1(\Ss)$, $H^2(\Ss)$ are represented by the following sets of cocycles respectively,
		\begin{description}
			\item[$H^0(\Ss)$] $\{1\}$,
			\item[$H^1(\Ss)$] $\{(x,-y,0)^T\}$,
			\item[$H^2(\Ss)$] $\{(0, \wt{zp'}, -xz, zy)^T, (z^i, \wt{z^i},0,0)^T\,|\, 0\leq i<n, i\neq n-\ell\}$.
		\end{description}
		When $p\sim z$, the bases for $H^0(\Ss)$, $H^1(\Ss)$, $H^2(\Ss)$ are represented by the following sets of cocycles respectively,
		\begin{description}
			\item[$H^0(\Ss)$]  $\{1\}$,
			\item[$H^1(\Ss)$]  $\{(x, -y, 0)^T, (x, y, 2z)^T\}$,
			\item[$H^2(\Ss)$]  $\{(0, \wt{p}, -xz, zy)^T, (1, 1, 0, 0)^T\}$.
		\end{description}
\end{theorem}

\begin{proof}	
	As we mentioned before the theorem, the statements concerning $H^0(\Ss)$ and $H^1(\Ss)$ can be obtained by lifting spectral sequences. Now we focus on the statements concerning $H^2(\Ss)$.   As a consequence of \cite[Prop.\ 6.2]{Solotar:Hochschild-homology-GWA-quantum}, one may choose a cocycle $U=(U_1, U_2, xU_3, U_4y)^T$ for some polynomials $U_j$ in $z$ as a representative of any second cohomological class. Observing the third column of the matrix $d^1$, we see that there exists a coboundary of the form $(V_1, V_2, xU_3, V_4y)^T$ for any given $U_3=u_0+u_2z^2+u_3z^3+\cdots$. This fact enables us to choose $U$ to be $(U_1, U_2, -kxz, U_4y)^T$ further for some scalar $k$. 
	Since
	\[
	d^2(U)=
	\begin{bmatrix}
	xU_1 -U_2 x+k \wt{p_1}xz p_2 \\
	-U_1 y + yU_2  -p_1 U_4y\wt{p_2} \\
	zU_1-U_1z  +kyxz -U_4 yx \\
	\wt{z}U_2-U_2\wt{z} +kxzy  -xU_4y
	\end{bmatrix}
	=
	\begin{bmatrix}
	\wt{U_1}x -U_2 x+k  \wt{z p'}x \\
	-y\wt{U_1} + yU_2  -y\wt{U_4 p'} \\
	kzp -U_4 p \\
	k\wt{zp}  -\wt{U_4p}
	\end{bmatrix},
	\]
	we conclude that $U$ is a cocycle if and only if $U_2=\wt{U_1}+k\wt{zp'}$, $U_4=kz$. Therefore, cocycles of the form $(U_1, \wt{U_1}+k\wt{zp'}, -kxz, kzy)^T$ represent all second cohomological classes.
	
	By the third row of $d^1$, we know $(0, \wt{zp'}, -xz, zy)^T\in \Ker d^2\setminus \Ima d^1$. Thus we have to judge whether a cocycle of the form $(U_1, \wt{U_1}, 0, 0)^T$ is a coboundary. Notice that
	\[
	d^1\bigl((0, z^jy, 0)^T\bigr)=(z^jp, \wt{z^jp},0, 0)^T=\sum_{i=0}^{n}a_i(z^{n-i+j}, \wt{z^{n-i+j}}, 0, 0)^T
	\]
	is a coboundary. Thus
	\[
	(z^{n+j}, \wt{z^{n+j}}, 0, 0)^T \in \sum_{i=0}^{n+j-1}\kk(z^i, \wt{z^i}, 0, 0)^T +\Ima d^1
	\]
	for all $j\in\nan$, and 
	\[
	(z^{j}, \wt{z^{j}}, 0, 0)^T \in \sum_{i=0}^{n}\kk(z^i, \wt{z^i}, 0, 0)^T +\Ima d^1
	\]
	for all $j\in\nan$.
	
	Denote by $\mathfrak{u}^i$ the cohomological class represented by $(z^i, \wt{z^i}, 0, 0)^T$. We obtain from the above computation that
	\begin{equation}\label{eq:HH2-1}
	\sum_{i=0}^{n}a_i\mathfrak{u}^{n-i}=0.
	\end{equation}
	Also, by $d^1\bigl((0, 0, -z)^T\bigr)=(zp', \wt{zp'}, 0, 0)^T$ we know
	\begin{equation}\label{eq:HH2-2}
	\sum_{i=0}^{n-1}(n-i)a_i\mathfrak{u}^{n-i}=0.
	\end{equation}
	Combining \eqref{eq:HH2-1} and \eqref{eq:HH2-2}, $(\mathfrak{u}^{n}, \mathfrak{u}^{n-1}, \dots, \mathfrak{u}^{0})^T$ satisfies the homogeneous system
	\[
	\begin{bmatrix}
	a_0 & a_1 & \cdots & a_{n-1} & a_n \\ na_0 & (n-1)a_1 & \cdots & a_{n-1} & 0
	\end{bmatrix}
	X=
	\begin{bmatrix}
	0 \\ 0
	\end{bmatrix}.
	\]
	Therefore, adding $(0, \wt{zp'}, -xz, zy)^T$ to a maximal linearly independent subset of $\{\mathfrak{u}^{n}, \mathfrak{u}^{n-1}, \dots, \mathfrak{u}^{0}\}$ will get a basis for $H^2(\Ss)$.
	
	The coefficient matrix of the homogeneous system is row-similar to 
	\[
	\begin{bmatrix}
	a_0 & a_1 & \cdots & a_{n-1} & a_n \\ 0 & a_1 & \cdots & (n-1)a_{n-1} & na_n
	\end{bmatrix}.
	\]
	Recall $\ell=\min\{ j \,|\, ja_j\neq 0 \}$, and it is not hard to see that the required maximal linearly independent subset is $\{\mathfrak{u}^{n-1},\ldots,\mathfrak{u}^{n-\ell+1}, \mathfrak{u}^{n-\ell-1},\ldots, \mathfrak{u}^{0}\}$ if $p\nsim z$, and is $\{\mathfrak{u}^{0}\}$  if $p\sim z$.
\end{proof}

Next, we will transfer the cohomology of $\Ss$ to the homology of $\TT$ via the cochain map $g$ given in Section \ref{sec:vdb-duality}. To this end,  let us write $\Phom{a}_r$ for the homology class which is represented by an $r$-cycle $a$ of $\TT$.

\begin{proposition}\label{prop:Per-homo}
	If $p\nsim z$, the nontrivial homological groups of $\TT$ are:
		\begin{enumerate}
			\item $H_0(\TT)=\kk\Phom{z}_0\oplus \bigoplus_{0\leq i< n, i\neq n-\ell} \kk\Phom{q^{-1}\wt{z^i \be}-z^i\be}_0$,
			\item $H_1(\TT)=\kk\Phom{(0, 0, 1)}_1$,
			\item $H_2(\TT)=\kk\Phom{(1\ot \be , -q^{-1}\ot \wt{\be}, q^{-1}\ot \al y, 0)}_2$.
		\end{enumerate}
		If $p\sim z$, the nontrivial homological groups of $\TT$ are:
		\begin{enumerate}
			\item $H_0(\TT)=\kk\Phom{z}_0\oplus \kk\Phom{1}_0$,
			\item $H_1(\TT)=\kk\Phom{(0, 0, 1)}_1\oplus \kk\Phom{(-q^{-1}y, x, 0)}_1$,
			\item $H_2(\TT)=\kk\Phom{(1\ot 1, -q^{-1}\ot 1, 0, 0)}_2$.
		\end{enumerate}
\end{proposition}

\begin{proof}
	(1) By applying $g$ to the representatives given by Theorem  \ref{thm:Per-cohomo}, one has the basis for the homology of $\TT[-2]$. When $p\nsim z$, a direct computation shows that
	\begin{align*}
	g^2\bigl((0, \wt{zp'}, -xz, zy)^T\bigr)
	&=q^{-1}\wt{z\be p'}+q^{-1}xz\al y =z, \\
	g^2\bigl((	z^i , \wt{z^i} , 0 , 0)^T\bigr)
	&=q^{-1}\wt{z^i \be}-z^i\be, \\
	g^1\bigl(( x , -y , 0)^T\bigr)
	&=(q^{-1}y\wt{\be} , x\be , x\al y)^T, \\
	g^0(1)&=(1\ot \be,	-q^{-1}\ot \wt{\be}, q^{-1}\ot \al y,	0 )^T.
	\end{align*} 
	After transposing all the vectors, we immediately obtain the assertions (1) and (3) in the case $p\nsim z$. For the assertion (2), we have to check
	\[
	(q^{-1}y\wt{\be} , x\be , x\al y)^T-(0, 0, 1)^T \in \Ima \pl^0_{[-2]}.
	\]
	In fact, this is a conclusion of
	\begin{align*}
	&\varphantom{=}
	\begin{bmatrix}
	q^{-1}y\wt{\be} \\ x\be \\ x\al y
	\end{bmatrix}-
	\begin{bmatrix}
	0 \\ 0 \\ 1
	\end{bmatrix}
	=\begin{bmatrix}
	q^{-1}y\wt{\be} \\ \wt{\be}x \\ -\wt{\be p'}
	\end{bmatrix} \\
	&=
	\begin{bmatrix}
	q^{-1}\ot y & y\ot 1 & 1\ot \wt{z}-z\ot 1 & 0 \\
	x\ot 1 & q\ot x & 0 & 1\ot z-\wt{z}\ot 1 \\
	-p_1\ot p_2 & -q\wt{p_1}\ot\wt{p_2}& x\ot q-q\ot x & y\ot 1-1\ot y
	\end{bmatrix}
	\begin{bmatrix}
	0 \\ q^{-1}\wt{\be} \\ 0 \\ 0
	\end{bmatrix}.
	\end{align*}
	
	For the situation $p\sim z$, the arguments are the same. The only point calling for special attention is that we may choose $\al=0$ and $\be=1$ safely, by Remark \ref{rmk:g-independence}.
\end{proof}

\section{Computation for the BV algebra structures}\label{sec:computation-bv}


Let $A$, $\Ss$, $\TT$, etc.\ be as in the previous section. We have known two resolutions of $A$, the periodical resolution and the normalized bar resolution. Our aim is to compute the BV algebra structure on $\HH^*(A)$. So a comparison between the two resolutions is needed. In fact, we can construct the required comparison $\eta_*$ in lower degrees, by chasing the diagram
\[
\begin {tikzcd}
A^e \ar[d, "\eta_0"] & (A^e)^3 \ar[l, "d_1"']\ar[d, "\eta_1"] & (A^e)^4 \ar[l, "d_2"']\ar[d,  "\eta_2"] & \cdots \ar[l, "d_3"']\\
A\ot A   & A\ot\bA\ot A \ar[l, "b'_1"']  & A\ot\bA^{\ot 2}\ot A \ar[l, "b'_2"'] & \cdots \ar[l, "b'_3"']
\end {tikzcd}
\]
in which
\begin{alignat*}{2}
\eta_0\colon & A^e\longrightarrow A\ot A, & & 1\ot 1 \longmapsto 1[\,]1, \\
\eta_1\colon & (A^e)^3\longrightarrow A\ot \bA\ot A, & & (1\ot 1,0,0)\longmapsto 1[x]1, \\
& & & (0,1\ot 1,0)\longmapsto 1[y]1, \\
& & & (0,0,1\ot 1)\longmapsto 1[z]1, \\
\eta_2\colon & (A^e)^4\longrightarrow A\ot\bA^{\ot 2}\ot A, &\quad & (1\ot 1,0,0,0)\longmapsto 1[y, x] 1-1[p_1, z] p_2, \\
& & &  (0,1\ot 1,0,0)\longmapsto 1[x, y] 1-1[\wt{p_1}, \wt{z}] \wt{p_2}, \\
& & &  (0,0,1\ot 1,0)\longmapsto q[z, x] 1-1[x, z] 1, \\
& & &  (0,0,0,1\ot 1)\longmapsto 1[z, y] 1-q[y, z] 1.
\end{alignat*}

Notice that $\eta_*$ induces a quasi-isomorphism $\TT \to C_{\sbullet}(A, A^\nu)$, which will be still denoted by $\eta_*$, by abuse of notations. For any $r$-cycle $c$ of $C_{\sbullet}(A, A^\nu)$, the homology class represented by $c$ will be written as $\Hhom{c}_r$ from now on. Thus let us begin to compute the BV algebra structure.

\subsection{Case one: $p\nsim z$}\label{subsec:case-1}

First of all, it follows from the dimensions of the cohomological groups that in $\HH^*(A)$, $a\smallsmile b=0$ if $|a|$, $|b|\geq 1$. Next we will determine the operator $\De$.

\begin{lemma}\label{lem:lemma1}
	If $i\geq 1$, then $\Phom{(0, 0, z^i)}_1=0$, and hence $\Hhom{1[z^{i+1}]}_1=0$.
\end{lemma}

\begin{proof}
	Clearly, $(0, 0, z^i)$ is indeed a $1$-cycle of $\TT$, so $\Phom{(0, 0, z^i)}_1$ makes sense. In order to show $\Phom{(0, 0,z^i) }_1=0$, it is sufficient to verify $(0, 0, z^i)^T\in \Ima \partial^{-2}$. 
	
	Let $u=( -(1-q^i)^{-1}\be z^i , q^{-1}(1-q^i)^{-1}\wt{\be z^i} , 0 , (1-q^i)^{-1} x\al z^i )^T$, and we have
	\begin{align*}
	\partial^{-2}(u)
	&=
	\begin{bmatrix}
	q^{-1}\ot y & y\ot 1 & 1\ot \wt{z}-z\ot 1 & 0\\
	x\ot 1 & q\ot x & 0 & 1\ot z-\wt{z}\ot 1\\
	-p_1\ot p_2 & -q\wt{p_1}\ot \wt{p_2} & x\ot q-q\ot x & y\ot 1-1\ot y
	\end{bmatrix}
	\begin{bmatrix}
	-(1-q^i)^{-1}\be z^i \\ q^{-1}(1-q^i)^{-1}\wt{\be z^i}\\ 0\\ (1-q^i)^{-1}x\al z^i
	\end{bmatrix}\\ 
	&=
	\begin{bmatrix}
	-q^{-1}(1-q^i)^{-1}\be z^i y+q^{-1}(1-q^i)^{-1}y \wt{\be z^i }\\
	-(1-q^i)^{-1}x\be z^i+ (1-q^i)^{-1}\wt{\be z^i}x\\
	(1-q^i)^{-1}\be z^ip'-(1-q^i)^{-1}\wt{\be z^i p'}+(1-q^i)^{-1}\al z^i p-(1-q^i)^{-1}\wt{\al z^i p}
	\end{bmatrix} \\
	&=
	\begin{bmatrix}
	0\\0\\(1-q^i)^{-1}z^i-(1-q^i)^{-1}\wt{z^i}
	\end{bmatrix}\\
	&=
	\begin{bmatrix}
	0 \\ 0\\ z^i
	\end{bmatrix}.
	\end{align*}
	It follows that  $(0, 0, z^i)$ is a $1$-boundary, namely $\Phom{(0, 0,z^i) }_1=0$. Therefore, $\Hhom{1[ z^{i+1}]}_1=0$, as desired.
\end{proof}

\begin{proposition}\label{prop:hoch-homology1}
	The nontrivial homological groups $H_{*}(A, A^\nu)$ which are induced by the Van den Bergh duality $\HH^*(A)\cong H_{2-*}(A, A^\nu)$ in Theorem \ref{thm:quasi-inverse} are:
	\begin{enumerate}
		\item $H_0(A, A^\nu)=\bigoplus_{0\leq i< n} \kk\Hhom{z^i[\,]}_0$,
		\item $H_1(A, A^\nu)=\kk\Hhom{1[ z]}_1$,
		\item $H_2(A, A^\nu)=\kk\Hhom{\fc}_2$, where $\fc$ is $\be[ y, x]-q^{-1}\wt{\be}[x, y]-\be p_2[p_1, z]+q^{-1}\wt{\be p_2}[\wt{p_1}, \wt{z}]-q^{-1}\al y[x, z] +\al y[ z, x]$.
	\end{enumerate}
\end{proposition}

\begin{proof}
	Combining Proposition \ref{prop:Per-homo} and the expressions of $\eta_*$ given above, we compute the bases for the three nontrivial groups. Notice that the results for $H_1(A, A^\nu)$ and $H_2(A, A^\nu)$ are exactly the same as given in the proposition; however, the result for $H_0(A, A^\nu)$ is
	\[
	\bigl\{\Hhom{z}_0, \Hhom{q^{-1}\wt{z^i \be}-z^i\be}_0 \,\bigm|\, 0\leq i <n, i\neq n-\ell \bigr\}.
	\]
	since the only common basis element is $\Hhom{z}_0$, we have to show that the subspace  with basis
	\[
	\bigl\{\Hhom{q^{-1}\wt{z^i \be}-z^i\be}_0 \,\bigm|\, 0\leq i <n, i\neq n-\ell \bigr\}.
	\]
	also contains
	\[
	\bigl\{\Hhom{z^i[\,]}_0 \,\bigm|\, 0\leq i <n, i\neq 1 \bigr\}
	\]
	as another basis.
	
	Suppose $\be=\sum_ju_jz^j$. Then
	\[
	q^{-1}\wt{z^i \be}-z^i\be=\sum_j u_j (q^{i+j-1}-1) z^{i+j},
	\]
	whose coefficient of $z$ is equal to zero. So we conclude that
	\begin{equation}\label{eq:inclusion1}
	\bigoplus_{0\leq i <n, i\neq n-\ell}\kk \Hhom{q^{-1}\wt{z^i \be}-z^i\be}_0 \subset \sum_{j\in \mathbb{N}, j\neq 1} \kk\Hhom{z^j}_0.
	\end{equation}
	The first author has proved in \cite[\S 5.4]{L:gwa-def} that $\Hhom{z^j}_0$  is a linear combination of $\Hhom{1}_0$, $\Hhom{z^2}_0$, $\Hhom{z^3}_0$, $\dots$, $\Hhom{z^{n-1}}_0$ for all $j\geq n$.  This fact, together with \eqref{eq:inclusion1}, implies
	\[
	\bigoplus_{0\leq i <n, i\neq n-\ell}\kk \Hhom{q^{-1}\wt{z^i \be}-z^i\be}_0 \subset \sum_{0\leq j <n, j\neq 1} \kk\Hhom{z^j}_0.
	\]
	The inclusion is in fact an equality by comparing the dimensions of both sides, and furthermore, the right-hand side is a direct sum.
\end{proof}

\begin{remark}
	In Hochschild cohomology theory, the class $\Hhom{\fc}_2$, which corresponds to the cohomology class $1\in Z(A)=\HH^0(A)$, is called the fundamental class for the Van den Bergh duality. 
\end{remark}

We are so fortunate that all the basis elements in Proposition \ref{prop:hoch-homology1} belong to $C_{\sbullet}^{(1)}(A, A^\nu)$, and hence Theorem \ref{thm:kk-bv} works. Applying the Connes operator $\BB$ to these elements, we have
\begin{align*}
\BB \colon C^{(1)}_0(A, A^\nu) &\longrightarrow C^{(1)}_1(A, A^\nu) \\
z^i[\,] &\longmapsto 1[z^i], \quad 0\leq i< n, i\neq n-\ell, \\
\BB \colon C^{(1)}_1(A, A^\nu) &\longrightarrow C^{(1)}_2(A, A^\nu) \\
1[ z] &\longmapsto 1[1, z]-1[z, 1].
\end{align*}
We remind the reader that if $i=0$ then $1[z^i]=1[1]=1\ot \bar{1}=0$, and similarly $1[1, z]=1[z, 1]=0$. Moreover, by Lemma \ref{lem:lemma1}, $\Hhom{1[z^{i}]}_1=0$ for $2\leq i< n$. Thus $\BB(z[\,])=1[z]$ is the unique nontrivial equation.

Recall that in the proof of Theorem \ref{thm:Per-cohomo}, the cohomological class represented by $(z^i, \wt{z^i}, 0, 0)^T$ is denoted by $\mathfrak{u}^i$. For completeness, the classes represented by $1$, $(x,-y,0)^T$, $(0, \wt{zp'}, -xz, zy)^T$ are denoted by $1$, $\mathfrak{s}$, $\mathfrak{v}$, respectively.  Since $z[\,]$, $1[z]$ correspond to $\mathfrak{v}$, $\mathfrak{s}$, the unique nontrivial equation $\BB(z[\,])=1[z]$ gives rise to $\De(\mathfrak{v})=\mathfrak{s}$. All the foregoing results are summarized as

\begin{theorem}\label{thm:main-result-1}
	If $p\nsim z$, then $\HH^*(A)$ as a BV algebra has $\{1, \mathfrak{s}, \mathfrak{v}, \mathfrak{u}^i \,|\, 0\leq i< n, i\neq n-\ell \}$ as a basis, where $|1|=0$, $|\mathfrak{s}|=1$, $|\mathfrak{v}|=|\mathfrak{u}^i|=2$, and in addition,
	\begin{enumerate}
		\item $1$ is the identity of $\HH^*(A)$ with respect to the cup product, and the cup products of other pairs of basis elements are trivial,
		\item $\De(\mathfrak{v})=\mathfrak{s}$, and $\De$ acts on other basis elements trivially.
	\end{enumerate}
\end{theorem}

\begin{corollary}
	The Gerstenhaber bracket on $\HH^*(A)$ is trivial for all $p\nsim z$.
\end{corollary}

\begin{proof}
	This follows from Theorem \ref{thm:main-result-1} and the formula \eqref{eq:gerstenhaber-bracket-generator}.
\end{proof}

\subsection{Case two: $p\sim z$}

The argument used in subsection \ref{subsec:case-1} works for this case.

\begin{proposition}\label{prop:hoch-homology2}
	The nontrivial homological groups $H_{*}(A, A^\nu)$ which are induced by the Van den Bergh duality $\HH^*(A)\cong H_{2-*}(A, A^\nu)$ in Theorem \ref{thm:quasi-inverse} are:
	\begin{enumerate}
		\item $H_0(A, A^\nu)=\kk\Hhom{z[\,]}_0 \oplus \kk\Hhom{1[\,]}_0$,
		\item $H_1(A, A^\nu)=\kk\Hhom{1[z]}_1 \oplus \kk\Hhom{x [y]-q^{-1}y [x] }_1$,
		\item $H_2(A, A^\nu)=\kk\Hhom{1 [y, x]-q^{-1} [x, y]}_2$.
	\end{enumerate}
\end{proposition}

The proof of the foregoing proposition is similar to that of Proposition \ref{prop:hoch-homology1}, so we omit it here.

The basis elements in Proposition \ref{prop:hoch-homology2} come from $C_{\sbullet}^{(1)}(A, A^\nu)$ too. By Theorem \ref{thm:kk-bv}, we apply the Connes operator $\BB$ to them, and thus have
\begin{align*}
\BB \colon C^{(1)}_0(A, A^\nu) &\longrightarrow C^{(1)}_1(A, A^\nu) \\
z[\,] &\longmapsto 1[z], \\
1[\,] &\longmapsto 1[1]=0, \\
\BB \colon C^{(1)}_1(A, A^\nu) &\longrightarrow C^{(1)}_2(A, A^\nu) \\
1[ z] &\longmapsto 1[1, z]-1[z, 1]=0, \\
x [y]-q^{-1}y [x] &\longmapsto -2[y, x]+2q^{-1} [x, y].
\end{align*}
As before, denote by $\mathfrak{s}$, $\mathfrak{t}$, $\mathfrak{u}$, $\mathfrak{v}$ respectively the classes represented by $(x, -y, 0)^T$, $(x, y, 2z)^T$, $(1, 1, 0, 0)^T$, $(0, \wt{p}, -xz, zy)^T$ which appear in the second part of Theorem \ref{thm:Per-cohomo}. The only two nontrivial actions
\[
\BB(z[\,])=  1[z],\quad \BB(x [y]-q^{-1}y [x])= -2[y, x]+2q^{-1} [x, y]
\] 
correspond to 
\[
\De(\mathfrak{v})=\mathfrak{s}, \quad \De(\mathfrak{t})=-2.
\]
\begin{remark}
	The $1$-cocycle $(x, -y, 0)^T$ determines a derivation $\delta_1\colon A\to A$ given by
	\[
	\delta_1(y^kz^ix^j)=(-k+j)y^kz^ix^j, \quad kj=0,
	\]
	and the $1$-cocycle $(x,y,2z)^T$ determines a derivation $\delta_2\colon A\to A$ given by
	\[
	\delta_2(y^kz^ix^j)=(k+2i+j)y^kz^ix^j, \quad kj=0,
	\]
	by the comparison  constructed in \cite[\S5.1]{L:gwa-def}.
\end{remark}

In order to completely describe the BV algebra structure, we have to express $\mathfrak{s}\smallsmile\mathfrak{t}$ as a linear combination of $\mathfrak{u}$ and $\mathfrak{v}$. Notice that in this case $\eta_2$ is simplified to be 
\begin{align*}
\eta_2(1\ot 1,0,0,0)&= 1[y, x] 1, \\
\eta_2(0,1\ot 1,0,0)&= 1[x, y] 1, \\
\eta_2(0,0,1\ot 1,0)&= q[z, x] 1-1[x, z] 1, \\
\eta_2(0,0,0,1\ot 1)&= 1[z, y] 1-q[y, z] 1.
\end{align*}
We need to compute $(\delta_1\smallsmile\delta_2)(w)$ where $w$ is one of the right-hand sides of the above four equations. Since
\begin{align*}
(\delta_1\smallsmile\delta_2)(y,x)&=-\delta_1(y)\delta_2(x)=yx=p, \\
(\delta_1\smallsmile\delta_2)(x,y)&=-\delta_1(x)\delta_2(y)=-xy=-\wt{p}, \\
(\delta_1\smallsmile\delta_2)(z,x)&=-\delta_1(z)\delta_2(x)=0, \\
(\delta_1\smallsmile\delta_2)(x,z)&=-\delta_1(x)\delta_2(z)=-2xz, \\
(\delta_1\smallsmile\delta_2)(z,y)&=-\delta_1(z)\delta_2(y)=0, \\
(\delta_1\smallsmile\delta_2)(y,z)&=-\delta_1(y)\delta_2(z)=2yz,
\end{align*}
the four required values are $p$, $-\wt{p}$, $2xz$, $-2qyz$. They form a vector $(p, -\wt{p}, 2xz, -2qyz)$ which is equal to $\eta_*(\delta_1\smallsmile\delta_2)$ and represents the class $\mathfrak{s}\smallsmile\mathfrak{t}$. We have
\[
(p, -\wt{p}, 2xz, -2qyz)^T=-2(0, \wt{p}, -xz, zy)^T+(p, \wt{p}, 0, 0)^T,
\] 
and
\[
\begin{bmatrix}
p \\ \wt{p} \\ 0 \\ 0
\end{bmatrix}=
\begin{bmatrix}
zp' \\ \wt{zp'} \\ 0 \\ 0
\end{bmatrix}=
\begin{bmatrix}
y\ot 1 & 1\ot x & -p_1\ot p_2 \\
1\ot y & x\ot 1 & -q\wt{p_1}\ot \wt{p_2} \\
\wt{z}\ot 1-1\ot z & 0 & q\ot x-x\ot 1 \\
0 & z\ot 1-1\ot \wt{z} & 1\ot y-y\ot q
\end{bmatrix}
\begin{bmatrix}
0 \\ 0 \\ -z
\end{bmatrix}
\]
is a $2$-coboundary. Hence $\mathfrak{s}\smallsmile\mathfrak{t}$ is represented by $-2(0, \wt{p}, -xz, zy)^T$, namely, $\mathfrak{s}\smallsmile\mathfrak{t}=-2\mathfrak{v}$.

Summarizing, we obtain

\begin{theorem}\label{thm:main-result-2}
	If $p\sim z$, then $\HH^*(A)$ as a BV algebra has $\{1, \mathfrak{s}, \mathfrak{t}, \mathfrak{u}, \mathfrak{v} \}$ as a basis, where $|1|=0$, $|\mathfrak{s}|=|\mathfrak{t}|=1$, $|\mathfrak{u}|=|\mathfrak{v}|=2$, and in addition,
	\begin{enumerate}
		\item $1$ is the identity of $\HH^*(A)$ with respect to the cup product, and the cup product is zero except $\mathfrak{s}\smallsmile\mathfrak{t}=-\mathfrak{t}\smallsmile\mathfrak{s}=-2\mathfrak{v}$,
		\item $\De(\mathfrak{v})=\mathfrak{s}$, $\De(\mathfrak{t})=-2$, and $\De$ acts on other basis elements trivially.
	\end{enumerate}
\end{theorem}

\begin{corollary}
	The Gerstenhaber bracket on $\HH^*(A)$ is trivial if $p\sim z$.
\end{corollary}

\begin{proof}
	This follows from Theorem \ref{thm:main-result-2} and the formula \eqref{eq:gerstenhaber-bracket-generator}.
\end{proof}

At the end of this section, we point out that the zeroth and the second Hochschild cohomological groups are  the only nontrivial groups for all generalized Weyl algebras of classical type (cf.\ \cite[Thm 1.2]{Farinati:Hochschid-homology-GWA}), and consequently the Hochschild cohomology admits trivial BV algebra structure. This is why we do not consider classical type in this paper.

\section{Applications}\label{sec:applications}

In this section, let us consider two concrete algebras arising from mathematical physics. The base field $\kk$ is fixed to be the complex number field $\mathbb{C}$, and $q$ is transcendental over rational numbers $\mathbb{Q}$.

Quantum weighted projective lines are related to the quantum group $SU_q(2)$. As an algebra, $SU_q(2)$ is generated by four elements $a$, $b$, $c$, $d$, subject to the relations
\begin{gather*}
ab=qba, \; ac=qca, \; bc=cb, \; bd=qdb, \; cd=qdc, \\
ad-da=(q-q^{-1})bc, \; ad-qbc=1.
\end{gather*}
Choose coprime positive integers $k$ and $l$, and endow the four generators with gradings as follows:
\[
|a|=k, \; |b|=l, \; |c|=-l, \; |d|=-k.
\]
Thus $SU_q(2)$ is made into a $\inn$-graded algebra. The homogeneous component of degree zero is called a quantum weighted projective line with respect to $(k,l)$, and is denoted by $\WL_q(k,l)$. For the background in mathematical physics of quantum weighted projective lines, we refer to \cite{Brzezinski-Fairfax:quantum-teardrop}.

\begin{theorem}[{\cite[Thm 2.1]{Brzezinski-Fairfax:quantum-teardrop}}]
	As a subalgebra of $SU_q(2)$, $\WL_q(k,l)$is generated by $a^lc^k$, $(-q)^kb^kd^l$, $-qbc$. If  we let
	\[
	x=a^lc^k, \; y=(-q)^kb^kd^l, \; z=-qbc,
	\]
	then the generating relations are
	\[
	xz=q^{2l}zx, \; yz=q^{-2l}zy, \; yx=z^k\prod_{i=1}^l(1-q^{-2i}z), \; xy=q^{2kl}z^k\prod_{i=0}^{l-1}(1-q^{2i}z).
	\]
\end{theorem}

By Definition \ref{def:gwa}, $\WL_q(k,l)$ is a generalized Weyl algebra, whose defining polynomial
\[
p=p(z)=z^k\prod_{i=1}^l(1-q^{-2i}z)
\]
has no multiple roots if and only if $k=1$. So let us consider the case $k=1$. The degree of $p$ is $l+1$, and $\WL_q(1,l)$ is skew Calabi--Yau with Nakayama automorphism $\nu$ given by
\[
\nu(x)=q^{2l}x, \; \nu(y)=q^{-2l}y, \; \nu(z)=z.
\]
Since $q$ is transcendental over $\mathbb{Q}$, the coefficients of $p$ are nonzero except the constant term. We thus have $\ell=1$.

After applying Theorem \ref{thm:main-result-1}  to $\WL_q(1,l)$, we get

\begin{proposition}
	The Hochschild cohomology of $\WL_q(1,l)$ is an $(l+3)$-dimensional graded space with basis $\{1, \mathfrak{s}, \mathfrak{v}, \mathfrak{u}^i \,|\, 0\leq i\leq l-1 \}$, where $|1|=0$, $|\mathfrak{s}|=1$, $|\mathfrak{v}|=|\mathfrak{u}^i|=2$. In addition,
	\begin{enumerate}
		\item $1$ is the identity of $\HH^*(\WL_q(1,l))$ with respect to the cup product, and the cup products of other pairs of basis elements are trivial,
		\item $\De(\mathfrak{v})=\mathfrak{s}$, and $\De$ acts on other basis elements trivially.
	\end{enumerate}
\end{proposition}

Another application is the structure for Podle\'s quantum spheres. Choose complex number $u$, $v$ such that $u^2+4v\neq0$. The subalgebra of $SU_q(2)$ that is generated by
\begin{align*}
x&=-q^{-1} (u ac-qc^2+va^2), \\
y&=u bd-qd^2+vb^2, \\
z&= q^{-1}(u bc-cd+vq^{-1}ab),
\end{align*}
is called the Podle\'s quantum sphere with parameters $u$, $v$, denoted by $\SSS^2_q(u,v)$. It turns out that the generating relations are
\begin{align*}
xz&=q^2zx, \quad yz=q^{-2}zy, \\
yx&=-z^2-uz+v, \\ 
xy&=-q^4z^2-uq^2z+v.
\end{align*}

\begin{remark}
	As the name suggests, Podle\'s quantum spheres were originally constructed by  Podle\'s \cite{Podles:quantum-sphere}. They are not only subalgebras of $SU_q(2)$, but also right coideal of it. In the literature, Podle\'s quantum spheres are said to be quantum homogeneous spaces of $SU_q(2)$.
\end{remark}

\begin{remark}
	Obviously, $\SSS_q(u, 0)\cong \WL_q(1,1)$. The latter is the same as $\mathbb{CP}_q^1$, the so-called complex quantum projective line.
\end{remark}

Podle\'s quantum spheres are all generalized Weyl algebras. Since the discriminant of the defining polynomial $p=-z^2-uz+v$ is $u^2+4v\neq 0$, $\SSS^2_q(u,v)$ is skew Calabi--Yau whose Nakayama automorphism $\nu$ satisfies
\[
\nu(x)=q^{2}x, \; \nu(y)=q^{-2}y, \; \nu(z)=z.
\]
It is easy to see that $\ell=1$ if $u\neq 0$, and $\ell=2$ if $u=0$.

Applying  Theorem \ref{thm:main-result-1} to $\SSS^2_q(u,v)$, we then obtain

\begin{proposition}
	$\HH^*(\SSS^2_q(u,v))$ is a $4$ dimensional graded space with basis $\{1, \mathfrak{s}, \mathfrak{v}, \mathfrak{u}^\star\}$, where
	\[
	\mathfrak{u}^\star=\begin{cases}
	\mathfrak{u}^0, & u\neq 0, \\ \mathfrak{u}^1, & u=0,
	\end{cases}
	\]
	and $|1|=0$, $|\mathfrak{s}|=1$, $|\mathfrak{v}|=|\mathfrak{u}^\star|=2$. Furthermore,
	\begin{enumerate}
		\item $1$ is the identity of $\HH^*(\SSS^2_q(u,v))$ with respect to the cup product, and the cup products of other pairs of basis elements are trivial,
		\item $\De(\mathfrak{v})=\mathfrak{s}$, and $\De$ acts on other basis elements trivially.
	\end{enumerate}
\end{proposition}

\begin{remark}
	According to this proposition, the basis for $\HH^*(\SSS^2_q(u,v))$ depends on if $u$ is zero, although the dimension is always four. 
\end{remark}

\begin{remark}
	When $u=0$, $v$ is impossible to be zero. In this case, one can directly check $\SSS^2_q(0,v)$ is isomorphic to $\SSS^2_q(0,1)$, which is called the equatorial Podle\'s quantum sphere.
\end{remark}

\section*{Acknowledgments}
Both authors are very grateful to the anonymous referees for their valuable comments, and to Guodong Zhou for pointing out an inaccuracy in the early version of this article.  This research is supported by the Natural Science Foundation of China No.\ 11971418.

\end{document}